\documentclass[12pt,twoside]{amsart}
\usepackage{amssymb,amsmath,amsfonts,amsthm}
\usepackage{mathrsfs}
\usepackage[X2,T1]{fontenc}
\usepackage[applemac]{inputenc}
\usepackage{cite}
\usepackage{latexsym}
\usepackage[dvips]{graphicx}
\usepackage{comment}
\usepackage{ulem}
\def\Im{\mathop{\rm Im}\nolimits}
\def\Re{\mathop{\rm Re}\nolimits}
\def\supp{\mathop{\rm supp}\nolimits}

\def\Im{\mathop{\rm Im}\nolimits}
\def\Re{\mathop{\rm Re}\nolimits}
\def\supp{\mathop{\rm supp}\nolimits}

\def\R{\mathbb R}
\def\C{\mathbb C}
\def\N{\mathbb N}

\def\ds{\displaystyle}

\newcommand{\afrac}[2]{\genfrac{}{}{0pt}{1}{#1}{#2}}
\newcommand\dslash{d\llap {\raisebox{.9ex}{$\scriptstyle-\!$}}}
\usepackage{colortbl}

\newcommand{\beqsn}{\arraycolsep1.5pt\begin{eqnarray*}}
	\newcommand{\eeqsn}{\end{eqnarray*}\arraycolsep5pt}
\newcommand{\beqs}{\arraycolsep1.5pt\begin{eqnarray}}
	\newcommand{\eeqs}{\end{eqnarray}\arraycolsep5pt}

\newtheorem{Th}{Theorem}[section]
\newtheorem{Rem}[Th]{Remark}
\newtheorem{Ex}[Th]{Example}
\newtheorem{Lemma}[Th]{Lemma}
\newtheorem{Def}[Th]{Definition}
\newtheorem{Prop}[Th]{Proposition}
\newtheorem{Cor}[Th]{Corollary}
	
\def\pxi{\langle \xi \rangle}
\def\px{\langle x \rangle}

\catcode`\@=11

\renewcommand{\section}%
{\setcounter{equation}{0}\@startsection {section}{1}{\z@}{-3.5ex plus -1ex
		minus -.2ex}{2.3ex plus .2ex}{\Large\bf}}

\title{KdV-type equations in projective Gevrey classes}
\author[Arias]{Alexandre Arias Junior}
\address{Alexandre Arias Junior\\
	Department of Computing and Mathematics \\ Universidade de S\~ao Paulo\\
	Ribeir\~ao Preto\\
	Brasil}
\author[Ascanelli]{Alessia Ascanelli}
\address{Alessia Ascanelli\\
	Dipartimento di Matematica ed Informatica\\Universit\`a di Ferrara\\
	Via Machiavelli 30\\
	44121 Ferrara\\
	Italy}
\email{alessia.ascanelli@unife.it}
\author[Cappiello]{Marco Cappiello}
\address{Marco Cappiello\\
	Dipartimento di Matematica ``G. Peano" \\Universit\`a di Torino\\
	Via Carlo Alberto 10\\
	10123 Torino\\
	Italy}
\email{marco.cappiello@unito.it}

\begin{document}
	
	\def\thefootnote{}
	\footnote{The first author was supported by S\~ao Paulo Research Foundation (FAPESP), grant 2022/01712-3. The second and third author were supported by the second author's INdAM-GNAMPA Project 2020.}
	
	\begin{abstract} We prove well-posedness of the Cauchy problem for a class of third order quasilinear evolution equations with variable coefficients in projective Gevrey spaces. The class considered is connected with several equations in Mathematical Physics as the KdV and KdVB equation and some of their many generalizations. 
	\end{abstract}
	
	\maketitle
	
	\markboth{\sc KdV-type equations in projective Gevrey classes}{\sc A. Arias Junior, A.~Ascanelli, M. Cappiello}

	\section{Introduction and main result}
	The	Korteweg-de Vries equation 
\begin{equation}\label{KdVequation}\partial_t u + \frac12 \sqrt{\frac{g}{h}} \sigma \partial_{x}^3u + \sqrt{\frac{g}{h}} \left( \alpha+ \frac32 u \right) \partial_x u = 0, \qquad t \in \R, x\in\R,
	\end{equation}
has been introduced in \cite{KdV}
to describe the wave motion in shallow waters; $u(t,x)$ represents the wave elevation, $h$ is the (constant) water level, $g$ the gravity, $\alpha$ a (small) constant, $\sigma=\frac{h^3}{3}-\frac{Th}{\rho g}$, $T$ describes the surface tension and $\rho$ the water density. It is the most famous example of dispersive third order evolution equation with (real) constant coefficients. Denoting $D=-i\partial,$ the equation \eqref{KdVequation} can be written in the form $P(u,D_t,D_x)u=0$, where 
\begin{equation}\label{KdVoperator}
	P(u, D_t,D_x)=D_t  - \frac12 \sqrt{\frac{g}{h}} \sigma D_{x}^3 + \sqrt{\frac{g}{h}} \left( \alpha+ \frac32 u \right) D_x.
	\end{equation} Notice that the principal symbol of $P$ (in the sense of Petrowski) is given by $$\sigma_{principal}(\tau,\xi):=\tau- \frac12 \sqrt{\frac{g}{h}} \sigma \xi^3$$ and admits the {\it{real}} characteristic root $\tau= \frac12 \sqrt{\frac{g}{h}} \sigma \xi^3$.
An operator of the form \eqref{KdVoperator} can be referred to as a quasilinear $3$-evolution operator, cf. \cite{ABbis, Mizohata}.
A huge number of variants of the equation \eqref{KdVequation} has been introduced and studied along the years to model different phenomena connected with the wave propagation, see for instance \cite{Kato, LP, TMI} and the references therein. One of these variants is the so-called KdV-Burgers (KdVB) equation, see \cite{Johnson, JM}, which appears for instance in the analysis of the flow of liquids containing gas bubbles and of the propagation of waves in an elastic tube containing a viscous fluid. The KdVB equation reads as follows
\begin{equation} \label{KdVBequation} \partial_t u + 2au\partial_x u +5b \partial^2_x u +c \partial_x^3 u=0,
	\end{equation} cf. \cite{JM}, where $a,b,c$ are real constants.
The associated operator 
\begin{equation}\label{KdVBoperator}
	P(u,D_t,D_x)= D_t -cD_x^3 +5ibD_x^2 +2a uD_x
	\end{equation}
is again a semilinear $3$-evolution operator with constant coefficients. With respect to \eqref{KdVoperator}, the operator \eqref{KdVBoperator} admits complex-valued coefficients in the lower order terms. We recall that complex-valued coefficients naturally arise in the study of other evolution equations of physical interest (think for instance to the Euler-Bernoulli vibrating beam operator studied in \cite{beam}). We also observe that assuming the coefficients of the equations \eqref{KdVequation}, \eqref{KdVBequation}  to be constant is just a simplification; in principle some of the coefficients may depend on $t$ and/or $x$. \\
Starting from these considerations our aim is to consider a class of quasilinear $3$-evolution equations with variable coefficients connected with the previous physical models. Namely we shall consider the Cauchy problem for equations of the form $P(t,x,u,D_t,D_x)u=f(t,x)$ where
\begin{equation}\label{pi}
	P(t,x,u,D_t,D_x)= D_t+a_3(t)D_x^3 + \sum_{j=0}^2 a_j(t,x,u)D_x^j, \qquad (t, x) \in [0,T] \times \R, \end{equation}
and $f$ is an assigned function.
Before addressing the general problem, let us spend some words about the linear case, that is the case when the coefficients $a_j,j=0,1,2,$ do not depend on $u$. In this situation, we are led to consider the 
initial value problem
\begin{equation} \label{genCP} \begin{cases}
	P(t,x,D_t,D_x)u=f(t,x)  \\ u(0,x)=g(x) \end{cases}, \qquad (t,x) \in [0,T]\times \R,
	\end{equation}
where 
\begin{equation}\label{3evolop}	P(t,x,D_t,D_x)= D_t+a_3(t)D_x^3 + \sum_{j=0}^2 a_j(t,x)D_x^j.\end{equation}
When the coefficients $a_j, j=0,1,2,3,$ are all smooth and real-valued, the related Cauchy problem is well posed in $L^2$ and in Sobolev spaces $H^m$ for every $m \in \R$, whereas when $a_2(t,x)$ is complex-valued, in \cite{ascanelli_chiara_zanghirati_necessary_condition_for_H_infty_well_posedness_of_p_evo_equations} it was proved that if the Cauchy problem is well-posed in $H^\infty(\R)= \cap_{m \in \R} H^m(\R)$, then 
there exist $M,N>0$ such that $\forall \varrho>0$
\begin{equation}\label{neccondition}
\sup_{x\in\R}\hskip+0.1cm \min_{0\leq\tau\leq t\leq T}\int_{-\varrho}^\varrho 
\Im a_{2}(t,x+3 a_3(\tau)\theta)d\theta\leq M\log(1+\varrho)+N.
\end{equation}
On the other hand, by \cite{ABZ} we know that if there exists $C>0$ such that for every $(t,x)\in [0,T]\times\R$
\begin{equation}\label{decayconditionsHinfty}|\Im a_2(t,x)|\leq \frac{C}{\langle x \rangle} \quad {\rm {and}}\quad |\Im a_1(t,x)|+|\partial_x\Re a_2(t,x)|\leq  \frac{C}{\langle x \rangle^{1/2}},\end{equation}
with $\langle x \rangle = (1+|x|^2)^{1/2},$
then the Cauchy problem is well-posed in $H^\infty(\R)$ with a loss of derivatives. Namely, given $f(t),g\in H^s(\R)$ for some $s\in\R$, there exists a unique solution with values in $H^{s-\delta}(\R)$ for some suitable $\delta>0.$ This type of results has been also extended to general linear $p$-evolution operators of the form 
\begin{equation} \label{pevolutionoperator}
	 	P(t,x,D_t,D_x)= D_t+a_p(t)D_x^p + \sum_{j=0}^{p-1} a_j(t,x)D_x^j, \qquad (t,x) \in [0,T] \times \R,  
 \end{equation}
 where $p$ is a positive integer, see \cite{ABZ, ascanelli_chiara_zanghirati_necessary_condition_for_H_infty_well_posedness_of_p_evo_equations}.
In the recent paper \cite{AACgev}, we considered the Cauchy problem \eqref{genCP} for the operator \eqref{3evolop} under weaker decay conditions (compared to \eqref{decayconditionsHinfty}) on the second order terms. Namely, we replaced the decay of $|\Im a_2| $ in \eqref{decayconditionsHinfty} by a decay of type $\langle x \rangle^{-\sigma}$ for some $\sigma \in (1/2,1)$.
In this case, $H^\infty$ well-posedness is lost due to the violation of \eqref{neccondition}. However, in analogy with the case $p=2$ treated in \cite{CRe1, KB}, under suitable assumptions on the regularity of the coefficients, it is natural to study the Cauchy problem in the Gevrey-Sobolev spaces  
$$
H^{m}_{\rho; \theta} (\R) = \{ u \in \mathscr{S}'(\R) :  \langle D \rangle^{m} 
e^{\rho \langle D \rangle^{\frac{1}{\theta}}} u \in L^{2}(\R) \},\quad \theta \geq 1,\ m, \rho \in \R,
$$
where $\langle D \rangle^m$ and $e^{\rho \langle D \rangle^{\frac{1}{\theta}}}$ are the Fourier multipliers with symbols $\langle \xi \rangle^m$ and $e^{\rho \langle \xi \rangle^{\frac{1}{\theta}}}$ respectively.
These spaces are Hilbert spaces with the following inner product
$$\langle u, v \rangle_{H^{m}_{\rho;\theta}} = \, \langle\langle D \rangle^{m} e^{\rho\langle D \rangle^{\frac{1}{\theta}}}u, \langle D \rangle^{m} e^{\rho\langle D \rangle^{\frac{1}{\theta}}}v \rangle_{L^{2}}, \quad u, v \in H^{m}_{\rho;\theta}(\R).
$$
The spaces 
$$ \mathcal{H}^\infty_{\theta} (\R): = \bigcup_{\rho >0}H^m_{\rho; \theta}(\R), \quad {H}^\infty_{\theta} (\R): = \bigcap_{\rho >0}H^m_{\rho; \theta}(\R)$$ are related to Gevrey classes in the following sense: 
$$ G_0^\theta(\R) \subset  \mathcal{H}^\infty_{\theta} (\R) \subset G^\theta(\R),\quad \gamma_0^\theta(\R) \subset H^\infty_{\theta} (\R) \subset \gamma^\theta(\R),$$
where $G^\theta(\R)$ (respetively, $\gamma^\theta(\R)$) 
denotes the space of all smooth functions $f$ on $\R$ such that 
\begin{equation}
\label{gevestimate}
\sup_{\alpha \in \N^n} \sup_{x \in \R} h^{-|\alpha|} \alpha!^{-\theta} |\partial^\alpha f(x)| < +\infty
\end{equation} 
for some $h >0$ (resp., for every $h>0$), and $G_0^\theta(\R)$ (resp. $\gamma_0^\theta(\R)$) is the space of all compactly supported functions contained in $G^\theta(\R)$ (resp. $\gamma^\theta(\R)$).\\
In \cite{AACgev}, we proved that if
\begin{itemize}
	\item [(i)] $a_3 \in C([0,T]; \R)$ and there exists  $C_{a_3} > 0$ such that $|a_3(t)| \geq C_{a_3}\ \forall t \in [0,T]$,
	\item [(ii)] $a_j \in C([0,T]; G^{\theta_0}(\R))$, $\theta_0 > 1$, for $j = 0,1,2$,
	\item [(iii)] $\exists \sigma \in (\frac{1}{2}, 1)$ with $\theta_0 < \frac{1}{2(1-\sigma)}$ and $C_{a_2} > 0$ such that  $|\partial^{\beta}_{x} a_2(t,x)| \leq C^{\beta+1}_{a_2} \beta!^{\theta_0} \langle x \rangle^{-\sigma}$ for every $t \in [0,T], \, x  \in \R, \beta \in \N_{0}$,
	\item [(iv)] $\exists C_{a_1}$ such that $|\Im\, a_1(t,x)| \leq C_{a_1}  \langle x \rangle^{-\frac{\sigma}{2}}$ for every $t \in [0,T], \, x \in \R$,
\end{itemize}
then the Cauchy problem \eqref{genCP} for the operator \eqref{3evolop} is well-posed in $\mathcal{H}^\infty_{\theta}(\R)$ for every $\theta \in [\theta_0, \frac{1}{2(1-\sigma)})$. Moreover, the solution satisfies the energy estimate
\begin{equation}
	\label{energyestimate}
	\| u(t, \cdot )\|_{H^{m}_{\rho-\delta; \theta}}^2 \leq C \left( \| g \|^2_{H^{m}_{\rho; \theta}} + \int_0^t \| f (\tau, \cdot) \|^2_{H^{m}_{\rho; \theta}}\, d\tau \right),\quad t\in[0,T],
\end{equation}
\textit{for a suitable} $\delta \in (0, \rho)$.
More recently, we realized that also well-posedness in ${H}^\infty_{\theta}(\R)$ can be obtained with minor modifications in the proof of the latter result, and assuming the coefficients $a_j$ to satisfy suitable \textit{projective} Gevrey estimates; in this case, we can prove an energy estimate of the form \eqref{energyestimate} \textit{for every} $\delta \in (0,\rho)$,
 and by this estimate well-posedness in ${H}^\infty_{\theta}(\R)$ follows. The proof of this result is a particular case of Theorem \ref{iop} here below when the coefficients $a_j$ are independent of $u$, cf. Corollary \ref{iop2}.
\vskip+0.2cm
Going back to quasilinear equations, in
	 this paper we shall consider the Cauchy problem  
	\beqs\label{CP}
	\left\{\begin{array}{ll}
		P(t,x,u(t,x),D_t,D_x)u(t,x) = f(t,x),\quad (t,x)  \in [0,T] \times\R,
		\\
		u(0,x) = g(x),\quad x  \in \R,
	\end{array}\right.
	\eeqs
	for the operator
\eqref{pi} in the Gevrey setting described above. As far as we know,  there are only a few results concerning KdV-type equations with constant coefficients in Gevrey spaces, see \cite{DBHK, GHGP, Goubet, HHP}.
Due to the loss of regularity appearing in the linear case, it is not possible in general to deduce local well-posedness for the problem \eqref{CP} from the above mentioned results for linear equations via a standard fixed point argument but we need more sophisticated techniques. 
To prove our main result we use an approach inspired by the method proposed in \cite{DA} for hyperbolic equations and in \cite{ABbis} for $p$-evolution equations in the $H^\infty$ setting. Here we adapt this method to the Gevrey setting. The proof relies on the application of Nash-Moser inversion theorem and gives the existence of a unique solution $u$ of \eqref{CP} in $C^1([0,T^*], H^\infty_\theta(\R))$ for some $T^* \in [0,T]$ by solving the equivalent integral equation $Ju \equiv 0$ in $[0,T^*],$  where 
the map $J: C^1([0,T], H^\infty_\theta(\R))\rightarrow C^1([0,T], H^\infty_\theta(\R))$ is defined by \beqs\label{ju}J(u)=u-g+i\ds\int_0^t (Pu)(s)ds-i\int_0^tf(s)ds. \eeqs
This can be achieved by proving that $J$ is a locally invertible map.
The main reason to work in $H^\infty_\theta(\R)$ instead than in $\mathcal{H}^\infty_\theta(\R)$ is the following: Nash-Moser theorem applies in the category of tame Fr\'echet spaces and $H^\infty_\theta(\R)$, equipped with its natural topology, is such a space, whereas this is not the case for $\mathcal{H}^\infty_\theta(\R)$. 

In order to give a precise assumption on the regularity and decay of the coefficients, we need the following definition.
\begin{Def}\label{functionspace} 
	For $\theta_0 >1$ and $\tau\geq 0$ we denote by $\Gamma^{\theta_0, \tau} (\R\times\C)$ the space of all functions  $f(x,w)$ defined on $\R \times \C$ which are smooth in $x$ and holomorphic in $w$ and satisfy the following condition: for every $A >0$ and every compact set $K \subset \C$ there exists a constant $C_K > 0$ such that
	$$
	\sup_{\beta, \gamma \in \N_0} \sup_{x \in \R, w \in K}|\partial^{\beta}_{x} \partial^{\gamma}_{w} f(x,w)| C_K^{-\gamma} \gamma!^{-1} A^{-\beta} \beta!^{-\theta_0} \langle x \rangle^{\tau} < +\infty,
	$$
	where $\partial_{x}$ stands for a real derivative and $\partial_{w}$ stands for a complex derivative.
\end{Def}
	
	We recall the notion of convergence in $\Gamma^{\theta_0, \tau} (\R\times\C)$. For $\{f_j\}_{j \in \N_0} \subset \Gamma^{\theta_0, \tau} (\R\times\C)$ and $f \in \Gamma^{\theta_0, \tau} (\R\times\C)$ we have $f_j \to f$ in $\Gamma^{\theta_0, \tau} (\R\times\C)$  as $j \to \infty$, whenever for every $A >0$ and every compact $K$ there exists $C_K > 0$ such that
	$$
	\sup_{\beta, \gamma \in \N_0} \sup_{x \in \R, w \in K}|\partial^{\beta}_{x} \partial^{\gamma}_{w} \{f_j(x,w) - f(x,w)\}| C_K^{-\gamma} \gamma!^{-1} A^{-\beta} \beta!^{-\theta_0} \langle x \rangle^{\tau} \to 0, \quad \text{as} \, j\to \infty.
	$$ 
	 
	Now we are ready to state the main result of this manuscript.

	\begin{Th}\label{main}
		Let  $a_3 \in C([0,T]; \R)$ such that $|a_3(t)| \geq C_{a_3} > 0$ for some constant $C_{a_3}$ and for every $t\in [0,T]$. Let moreover $\sigma \in \left(\frac12,1\right)$ and $\theta_0<\frac1{2(1-\sigma)}$ and assume that for $j=0,1,2$ the coefficients $a_j \in C([0,T], \Gamma^{\theta_0, \frac{j\sigma}{2}}(\R \times \C))$. Then the Cauchy problem \eqref{CP} is locally in time well-posed in $H^\infty_\theta(\R)$ for every $\theta\in \left[\theta_0,\frac1{2(1-\sigma)}\right)$: namely for all $f\in C([0,T];H^\infty_\theta(\R))$ and $g\in H^\infty_\theta(\R)$, there exists $0<T^*\leq T$ and a unique solution $u\in C^1([0,T^*];H^\infty_\theta(\R))$ of \eqref{CP}.
	\end{Th}

	\begin{Ex}
		Simple examples of coefficients $a_j$ are given by $a_j(t,x,w)= a(t,x) \langle x\rangle^{-\frac{\sigma j}2} b(w)$ with $a \in C([0,T]; \gamma^{\theta_{0}}(\R))$ and  $b(w)=w^r$, $r\in\N$, or $b(w)=e^w$ or some other entire function. Indeed, given an entire function $h$, for every compact $K\subset \C$ there exists a positive constant $C_K$ such that for every $w \in K$ we have $|\partial^\alpha_w h(w)|\leq C_K^{\alpha+1}\alpha!$.
	\end{Ex} 

\begin{Rem}
	The result obtained in this paper concerns $3$-evolution equations in one space dimension as \eqref{KdVequation}, \eqref{KdVBequation}. The extension of this result to higher space dimension requires a major technical effort in the definition of the change of variable needed to study the linearized problem associated to \eqref{CP}. We will treat this extension in a future paper. 
\end{Rem}

The paper is organized as follows. In Section \ref{preliminaries} we recall some basic definitions and properties of tame Fr\'echet spaces and the statement of Nash-Moser theorem. Moreover, we prove that $H^\infty_\theta(\R)$ is a tame Fr\'echet space. Then, we introduce pseudodifferential operators of infinite order which are employed in the next sections to study the linearized Cauchy problem associated to \eqref{CP}. Section \ref{linearized} is devoted to the study of this linear problem which is done using similar techniques as the ones used in \cite{AACgev} adapted to the projective Gevrey setting. Finally, in Section \ref{finalsection} we apply Nash-Moser theorem to obtain local in time well-posedness of \eqref{CP}.
	
	\section{Preliminaries}
	\label{preliminaries}
	\subsection{Function spaces}
In this subsection we recall some basic facts concerning tame Fr\'echet spaces and prove that $H^\infty_\theta(\R)$ is such a space. Moreover, we recall the statement of Nash-Moser inversion theorem, see \cite{H}. 

\medskip
A {\it graded} Fr\'echet space $X$ is a Fr\'echet space endowed with a {\it grading}, i.e. an increasing sequence of semi-norms:
\beqsn
|x|_n\leq|x|_{n+1},\qquad \forall n\in\N_0, \, x\in X.
\eeqsn

\begin{Ex} Given a Banach space $B$, consider the space $\Sigma(B)$ of all sequences
	$\{v_k\}_{k\in\N_0}\subset B$ such that
	\beqsn
	|\{v_k\}|_n:=\left(\sum_{k=0}^{\infty}e^{2nk}\|v_k\|_B^2\right)^{1/2}<+\infty, \qquad
	\forall n\in\N_0. 
	\eeqsn
	We have that $\Sigma(B)$ is a graded Fr\'echet space with the topology induced by the family of seminorms $| \cdot |_n$ (which is in fact a grading on $\Sigma(B)$).
\end{Ex}

We say that a linear map $L:\ X\to Y$ between two graded Fr\'echet spaces is 
a {\it tame linear map} if there exist $r,n_0\in\N$ such that for every
integer $n\geq n_0$ there exists a constant $C_n>0$, depending only on $n$, such that
\beqs
\label{2.2DA}
|Lx|_n\leq C_n|x|_{n+r}, \qquad\forall x\in X.
\eeqs
The numbers $n_0$ and $r$ are called respectively {\it base} and {\it degree} of the
{\it tame estimate} \eqref{2.2DA}.

\begin{Def}\label{tamespace}
	A graded Fr\'echet space $X$ is said to be {\it tame} if there exist a Banach space $B$
	and two tame linear maps $L_1:\, X\to\Sigma(B)$ and $L_2:\,\Sigma(B)\to X$
	such that $L_2\circ L_1$ is the identity on $X$.
\end{Def}

Obviously, given a graded Fr\'echet space $X$ and a tame space $Y$, if there exist two linear tame maps $L_1:\, X\to Y$ and $L_2:\,Y\to X$ such that $L_2\circ L_1$ is the identity on $X$, then also $X$ is a tame space.

\begin{Th}\label{tameness}
	The space $H^\infty_{\theta}(\R^{n})$ is a tame Fr\'echet space.
\end{Th}
\begin{proof} As standard, we shall denote here and throughout the paper the Fourier transform of a function (or a distribution) $u$ by $\hat{u}$ or by $\mathcal{F}
(u).$ First of all, it is easy to verify that $H^\infty_{\theta}(\R^{n})$ is a graded Fr\'echet space with the increasing family of seminorms
	$$ |f|_{k} := \|f\|_{H^0_{k,\theta}}=\| e^{k \langle \cdot \rangle^{1/\theta}}\hat{f}(\cdot)\|_{L^2}, \qquad k=1,2,3,\ldots.$$
	Consider now the space $\Sigma(L^2(\R^{n}))$ and the map $L_1:H^\infty_{\theta}(\R^{n}) \to \Sigma(L^2(\R^{n}))$ defined as $L_1(f)=\{f_j\}, j=1,2,3,\ldots,$ where $f_j = \mathcal{F}^{-1}(\chi_j \hat{f})$ and the functions $\chi_j$ are such that $\chi_j(\xi)=1$ if $j^\theta \leq \pxi < (j+1)^\theta$ and $\chi_j(\xi)=0$ otherwise. Then we have 
	\begin{eqnarray*}
		|\{f_j\}|_k^2 &=& \sum_{j=1}^\infty e^{2jk} \| \hat{f}_j\|^2_{L^2} 
		= \sum_{j=1}^\infty e^{2jk} \| \chi_j \hat{f}\|^2_{L^2} \\
		&=& \sum_{j=1}^\infty e^{2jk} \int_{\R^n}| \chi_j(\xi)e^{-\rho\pxi^{1/\theta}} e^{\rho\pxi^{1/\theta}}\hat{f}(\xi)|^2\, d\xi \\
		&\leq & \sum_{j=1}^\infty e^{2j(k-\rho)}  \| e^{\rho\langle \cdot \rangle^{1/\theta}}\hat{f}(\cdot)\|^2_{L^2} \leq C_{k,\rho} \| f\|^2_{H^0_{\theta, \rho}}
	\end{eqnarray*}
	for every $\rho >k$. In particular, for $\rho=k+1$ we obtain that $|\{f_j\}|_k \leq C'_k | f|_{k+1}$, hence $L_1$ is a tame linear map.
	Similarly, we define the map $L_2: \Sigma(L^2(\R^{n})) \to H^\infty_\theta(\R^{n})$ as 
	$$L_2(\{f_j\})= \mathcal{F}^{-1}\left( \sum_{j=1}^\infty \chi_j\hat{f_j}\right).$$
	We have 
	\begin{eqnarray*}
		|L_2 \{f_j\}|_k^2 &=& \left\| e^{k\langle \cdot \rangle^{1/\theta}}\sum_{j=1}^\infty \chi_j(\cdot)\hat{f_j}(\cdot)\right\|^2_{L^2} 
		= \int_{\R^n}  \sum_{j=1}^\infty  e^{2k\langle \xi \rangle^{1/\theta}}|\chi_j(\xi)\hat{f_j}(\xi)|^2\, d\xi  \\
		&\leq& \sum_{j=1}^\infty e^{2k(j+1)}\int_{\R^n}|  \chi_j(\xi)\hat{f_j}(\xi)|^2\, d\xi 
		\leq  \sum_{j=1}^\infty   e^{2k(j+1)} \| \hat{f_j}\|^2_{L^2} =e^{2k}  |\{f_j\}|_k^2.
	\end{eqnarray*}  
	Hence, also $L_2$ is a tame linear map. Moreover, it is easy to verify that $L_2 \circ L_1$ is the identity map on $H^\infty_\theta(\R^{n})$.
\end{proof}

	\begin{Def}
		Let $X,Y$ be two graded spaces, $U$ be an open subset of $X$. A map $T:U\to Y$ is said to be tame if for every $u \in U$ there exist a neighborhood $U'$ of $u$, $r \geq 0$ and $n_0 \in \N$ such that for every $n \geq n_0$ there exists a constant $C_n>0$ such that
		$$|T(u)|_n \leq C_n (1+|u|_{n+r})$$ for all $u \in U'.$ The map $T$ is said to be smooth tame if $T$ is $\mathcal{C}^\infty$ and its derivatives $D^n T:U \times X^n \to Y$ are tame for every $n \in \N.$
	\end{Def}
	
	Finally, we recall the statement of Nash-Moser inversion theorem, cf. \cite{H}.
	
	\begin{Th}\label{NM}
		(Nash-Moser) Let $X,Y$ be tame Fr\'echet spaces, $U$ be an open subset of $X$ and let $T:U \to Y$ be a smooth tame map. If for every fixed $u \in U, h\in Y$ the equation $DT(u)v=h$ has a unique solution $v=S(u,h)$ and if the map $S:U\times Y \to X$ is smooth tame, then $T$ is locally invertible at any point and each local inverse is smooth tame.
	\end{Th}
	\subsection{Pseudodifferential operators }
	\label{pseudos} 
	In this subsection we introduce the pseudodifferential operators of infinite order which will be used to prove the well-posedness for the linearized Cauchy problem associated to \eqref{CP}. Although the arguments in the next sections concern one space dimensional problems, it is appropriate to introduce these operators in arbitrary dimension for future applications.
		
	Fixed $ \mu \geq 1$, $A>0$ and $m,m_1,m_2 \in \R$ we will consider the following Banach spaces: 
		$$p(x,\xi) \in S^{m}_{\mu}(\R^{2n};A) \iff 
	\sup_{\overset{\alpha, \beta \in \N_{0}^{n}}{x,\xi \in \R^{n}}} |\partial_\xi^\alpha \partial_x^\beta p(x,\xi)| A^{-|\alpha+\beta|} (\alpha! \beta!)^{-\mu} \langle \xi \rangle^{-m+|\alpha|} < +\infty,
	$$
	$$
	p(x,\xi) \in \tilde{S}^{m}_{\mu}(\R^{2n};A) \iff |p|_{A} := \sup_{\overset{\alpha, \beta \in \N_{0}^{n}}{x,\xi \in \R^{n}}} |\partial_\xi^\alpha \partial_x^\beta p(x,\xi)| A^{-|\alpha+\beta|} (\alpha! \beta!)^{-\mu} \langle \xi \rangle^{-m} < +\infty,
	$$
	$$ p \in \textrm{\textbf{SG}}^{m_1,m_2}_{\mu}(\R^{2n}; A) \Leftrightarrow 
	\sup_{\stackrel{\alpha, \beta \in \N^{n}_{0}}{x,\xi \in \R^{n}} }| \partial_\xi^\alpha \partial_x^\beta p(x,\xi) |A^{-|\alpha+\beta|} (\alpha!\beta!)^{-\mu} \pxi^{-m_1+|\alpha|} \px^{-m_2+|\beta|}  <+\infty.
	$$ 
We set  
	\begin{equation}\label{projsymb1}S^m_{\mu}(\R^{2n}):= \bigcup_{A>0}S^{m}_{\mu}(\R^{2n};A), \qquad \tilde{S}^m_{\mu}(\R^{2n}):= \bigcup_{A>0}\tilde{S}^{m}_{\mu}(\R^{2n};A),\end{equation}
	\begin{equation}\label{projsymb2} \textbf{\textrm{SG}}^{m_1,m_2}_{\mu}(\R^{2n}):= \bigcup_{A>0}\textrm{\bf SG}^{m_1,m_2}_{\mu}(\R^{2n};A) \end{equation} endowed with the inductive limit topology
	and 
	\begin{equation}\label{projsymb1}\Gamma^m_{\mu}(\R^{2n}):= \bigcap_{A>0}S^{m}_{\mu}(\R^{2n};A), \qquad \tilde{\Gamma}^m_{\mu}(\R^{2n}):= \bigcap_{A>0}\tilde{S}^{m}_{\mu}(\R^{2n};A),\end{equation}
	\begin{equation}\label{projsymb2}\Gamma G^{m_1,m_2}_{\mu}(\R^{2n}):= \bigcap_{A>0}\textrm{\bf SG}^{m_1,m_2}_{\mu}(\R^{2n};A) \end{equation}
	endowed with the projective limit topology. 

	\begin{Rem}
		We observe that if $\mu< \theta$, then obviously for every $A>0$ we have $S^m_{\mu}(\R^{2n};A) \subset \Gamma^m_\theta (\R^{2n})$, $\tilde{S}^m_{\mu}(\R^{2n};A) \subset \tilde{\Gamma}^m_\theta (\R^{2n})$ and $\textrm{\bf SG}^{m_1,m_2}_{\mu}(\R^{2n};A) \subset \Gamma G^{m_1,m_2}_{\theta}(\R^{2n})$.
	\end{Rem}
	
	Taking into account the latter remark, in the sequel we shall consider symbols satisfying the estimates above for a fixed constant $A>0$ as subsets of some projective symbol classes with a lower Gevrey regularity as in \eqref{projsymb1}, \eqref{projsymb2}. For this reason we shall state the next results only for this type of classes.
	
	For a given symbol $p \in \tilde{\Gamma}^{m}_{\theta}(\R^{2n})$ we denote by $p(x,D)$ or by $\textrm{op}(p)$ the pseudodifferential operator defined by 
	\begin{equation} \label{pseudop}
		p(x,D) u (x) = \int e^{i\xi x} p(x,\xi) \widehat{u}(\xi) \dslash\xi, \quad u \in \gamma_0^\theta(\R^{n}),
	\end{equation}
	where $ \dslash\xi = (2\pi)^{-n}d\xi.$
	Arguing as in \cite[Theorem 3.2.3]{Rodino_linear_partial_differential_operators_in_gevrey_spaces} or \cite[Theorem 2.4]{Zanghirati} it is easy to verify that operators of the form \eqref{pseudop} with symbols from $\tilde{\Gamma}_{\theta}(\R^{2n})$ map continuously $\gamma_0^\theta(\R^n)$ into $\gamma^\theta(\R^n)$. Moreover, from the classical theory of pseudodifferential operators, they extend to linear and continuous operators from $H^{m'}(\R^n)$ to $H^{m'-m}(\R^n)$. For our purposes, it is also important to state the action of these operators on the Gevrey-Sobolev spaces defined in the Introduction. The following result is a direct consequence of \cite[Proposition 6.3]{KN} applied to symbols from $\tilde{\Gamma}_\theta^m (\R^{2n})$.
	
	\begin{Prop}\label{prop_continuity_finite_order_gevrey_sobolev}
		Let $p \in \tilde{\Gamma}^{m}_{\theta}(\R^{2n})$. Then the operator $p(x,D)$ 
		maps continuously $H^{m'}_{\rho;\theta}(\R^n)$ into $H^{m'-m}_{\rho;\theta}(\R^n)$ for every $m', \rho \in \R$.
	\end{Prop}
	
	By \cite[Proposition 6.4]{KN}, given $p \in \Gamma^{m}_{\theta}(\R^{2n})$ and $q \in \Gamma^{m'}_{\theta}(\R^{2n})$, the operator $p(x,D)q(x,D)$ is a pseudodifferential operator with symbol $s$ given for every $N \geq 1$ by
	$$s(x,\xi)= \sum_{|\alpha| <N}(\alpha!)^{-1}\partial_\xi^\alpha p(x,\xi) D_x^\alpha q(x,\xi) +r_N(x,\xi),$$
	where $r_N \in \Gamma_\theta^{m+m'-N}(\R^{2n}).$

	In the following we shall consider also particular symbols of infinite order, that is growing exponentially at infinity. Such operators are frequently used in the analysis of evolution equations in the Gevrey setting, see for instance \cite{AAC3evolGelfand-Shilov, AACgev, AscanelliCappielloHolder, ascanellicappielloJMPA, CM,  CRe1, KB, KN, Zanghirati}. In particular, in this paper they will be employed to define the change of variables which allows to treat the linearized problem associated to \eqref{CP}.
	We shall not develop a complete calculus for pseudodifferential operators of infinite order here since for our purposes we can limit ourselves to considering some particular examples of such operators, namely defined by a symbol of the form $e^{\Lambda(x,\xi)}$ for some $\Lambda \in S^{1/\kappa}_{\mu}(\R^{2n};A),$ where $\kappa> 1$ and $\Lambda$ is real-valued.
	It is easy to verify that $e^{\pm \Lambda}$ satisfies an estimate of the form
	\begin{equation}\label{expest}
		|\partial_\xi^\alpha \partial_x^\beta e^{\pm \Lambda(x,\xi)}| \leq  A_1^{|\alpha+\beta|}\pxi^{-|\alpha|}(\alpha!\beta!)^\mu e^{2\rho_0\langle \xi \rangle^{\frac1{\kappa}}}
	\end{equation}
	for some positive constant $A_1$ independent of $\alpha, \beta$, where 
	$$
	\rho_0:= \sup_{(\alpha, \beta) \in \N_{0}^{2n}} \sup\limits_{(x,\xi) \in \R^{2n}} A^{-|\alpha+\beta|} (\alpha!\beta!)^{-\mu} \langle \xi \rangle^{-1/\kappa + |\alpha|}|\partial^{\alpha}_{\xi}\partial^{\beta}_{x} \Lambda(x,\xi)|,
	$$ 
	see \cite[Lemma 6.2]{KN}. The estimate \eqref{expest} guarantees that the related pseudodifferential operator
	$$e^{\pm \Lambda}(x,D) u(x) = \int_{\R^n} e^{i\xi x \pm\Lambda(x,\xi)}\hat{u}(\xi)\, \dslash \xi$$ 
	is well defined and continuous as an operator from $\gamma_0^\theta(\R^n)$ to $\gamma^\theta(\R^n)$ for every $\theta \in (\mu, \kappa)$. 
	We shall also consider the so-called reverse operator of $e^{\pm \Lambda}(x,D)$, denoted by $^{R}\{e^{\pm \Lambda}(x,D)\}$. This operator, introduced in \cite[Proposition 2.13]{KW} as the transposed of $e^{\pm \Lambda}(x,-D)$,  is defined as an oscillatory integral by
	$$^{R}\{e^{\pm \Lambda}(x,D)\}u(x) = Os - \iint e^{i\xi (x-y) \pm \Lambda(y,\xi)}u(y)\, dy \dslash \xi.
	$$
	The following continuity result holds for the operators $e^\Lambda(x,D)$ and $^R\{e^\Lambda(x,D)\}$.

	\begin{Prop}\label{contgev}
		Let $\Lambda \in \tilde{S}_\mu^{1/\kappa}(\R^{2n};A)$ for some $A>0$ and $\kappa, \mu \in \R$ such that $1<\mu < \kappa$ and let  $\rho,m\in\R$ and $\theta  \in (\mu,\kappa)$. Then the operators $e^{\Lambda}(x,D)$ and $^R\{e^{\Lambda}(x,D)\}$ map continuously $H^{m}_{\rho; \theta} (\R^n)$ into $H^{m}_{\rho-\delta; \theta}(\R^n)$ for every $\delta >0$.	
	\end{Prop}
	\begin{proof}
		We observe that $e^{\Lambda}(x,D) = a(x,D) e^{\delta \langle D \rangle^{\frac{1}{\theta}}}$ for every $\delta > 0$, where $a(x,\xi) = e^{\Lambda(x,\xi)-\delta\langle \xi \rangle^{\frac{1}{\theta}}} $. Since $\mu < \theta < \kappa$ we easily obtain $a \in \tilde{\Gamma}^{0}_{\theta}(\R^{2n})$. So we get from Proposition \ref{prop_continuity_finite_order_gevrey_sobolev} that $e^{\Lambda}: H^{m}_{\rho; \theta} (\R^n) \to H^{m}_{\rho-\delta; \theta}(\R^n)$ continuously for every $m, \rho \in \R$. The continuity of $^R\{e^{\Lambda}(x,D)\}$ follows by similar arguments. 
	\end{proof}

	In the next result we shall need to work with the weight function $\langle\xi\rangle_h= (h^2+|\xi|^2)^{1/2}$
	where $h \geq 1$. We point out that we can replace $\langle \xi \rangle$ by $\langle \xi \rangle_{h}$ in all the previous definitions and statements, and this replacement does not change the dependence of the constants, that is, all the previous constants are independent of $h$. Moreover, we also need the following stronger hypothesis on $\Lambda(x,\xi):$
	\begin{equation}\label{equation_stronger_hypothesis_on_Lambda}
		|\partial_\xi^\alpha \partial_x^\beta \Lambda(x,\xi)| \leq C_{\Lambda}^{|\alpha+\beta|+1} \alpha!^{\mu}\beta!^{\mu} \langle \xi \rangle^{-|\alpha|}_{h}, 
	\end{equation}
	whenever $|\beta| \geq 1$. This means that $\partial_\xi^\alpha \partial_x^\beta \Lambda$ behaves like a symbol of order $0$ if $\beta \neq 0$. We will show in the next Section that this condition will be fulfilled by the symbol $\tilde{\Lambda}$ appearing in the change of variable.
	
	\begin{Th}\label{conjthmnew}
		Let $p $ be a symbol in $\Gamma^m_\theta (\R^{2n})$ and let $\Lambda$ satisfy, for some $C_{\Lambda}>0$ and $\mu < \theta < \kappa$:
		\begin{equation}\label{firstasslambda}
			|\partial_\xi^\alpha \Lambda(x,\xi)| \leq C_{\Lambda}^{|\alpha|+1} \alpha!^{\mu} \langle \xi \rangle^{\frac{1}{\kappa}-|\alpha|}_{h}
		\end{equation}
		and \eqref{equation_stronger_hypothesis_on_Lambda} for $\beta \neq 0$. 
		Then there exists $h_0 = h_0(C_{\Lambda}) \geq 1$ such that if $h \geq h_0$, then 
		\begin{multline*}
			e^\Lambda(x,D) p(x,D) ^R\{e^{-\Lambda}(x,D)\} = p(x,D)+ \textrm{op} \left( \sum_{1 \leq |\alpha+\beta| < N} \frac{1}{\alpha!\beta!} \partial^{\alpha}_{\xi} \{\partial^{\beta}_{\xi} e^{\Lambda(x,\xi)} D^{\beta}_{x}p(x,\xi) D^{\alpha}_{x}e^{-\Lambda(x,\xi)} \} \right)  \\ + r_{N}(x,D) + r_{\infty}(x,D),
		\end{multline*}
		where $r_N$ and $r_\infty$ satisfy the following conditions: there exists $c'= c'(\Lambda) > 0$ and for every $A>0$ there exists $C_A > 0$ such that 
		\begin{equation} \label{stimaperilresto}
			|\partial^{\alpha}_{\xi}\partial^{\beta}_{x}r_N(x,\xi)| \leq C_A A^{|\alpha+\beta|+2N}\alpha!^{\theta}\beta!^{\theta}N!^{2\theta-1} \langle \xi \rangle_h^{m-(1-\frac{1}{\kappa})N - |\alpha|},
		\end{equation}
		\begin{equation} \label{stimaperilresto2}
			|\partial^{\alpha}_{\xi}\partial^{\beta}_{x}r_{\infty}(x,\xi)| \leq C_{A} A^{|\alpha+\beta|+2N}\alpha!^{\theta}\beta!^{\theta}N!^{2\theta-1} e^{-c' \langle \xi \rangle_h^{\frac{1}{\theta}}}.
		\end{equation}
	\end{Th}
	
	\begin{Rem}
		Notice that choosing $N$ sufficiently large depending on $\kappa$, we can consider $r_N$ as a symbol of order $0$. Concerning the remainder term $r_\infty$, it is easy to verify that the corresponding operator possesses regularizing properties in Gevrey classes, namely it maps $(G_0^\theta)'(\R^n)$ into $G^\theta(\R^n)$. However, to prove our results it will be sufficient to regard also $r_\infty$ as a symbol of $\tilde\Gamma^0_\theta(\R^{2n}).$ In conclusion, in the computations of Section 4, choosing $N$ large enough, we shall always consider the remainder term $r_N+r_\infty$ as a symbol of $\tilde\Gamma^0_\theta(\R^{2n})$ and apply to it Proposition \ref{secondconjthm} below.
	\end{Rem}
	
	\begin{Rem} 
			The proof of Theorem \ref{conjthmnew} follows by applying
			readily in the projective Gevrey setting the same argument used in the proofs of of Theorems $6.9$ and $6.10$ of \cite{KN} and Theorem $2$ of \cite{AACgev} in the classical Gevrey framework. For this reason, we omit it for the sake of brevity. We just stress the fact that dealing now with projective Gevrey regular symbols $p$,   it is possible to conclude that the remainders $r_N$ and $r_{\infty}$ also satisfy this type of estimates, cf. \eqref{stimaperilresto}, \eqref{stimaperilresto2}.    
		\end{Rem}
	
	\noindent 
	Now we consider the conjugation with an operator of the form $e^{\Lambda_{\rho',k}}(t,D)$ where $\Lambda_{\rho',k}(t,\xi)= \rho' \langle \xi \rangle^{1/\theta}+k(T-t)\langle \xi \rangle^{2(1-\sigma)}$ for some $\rho' \in (0,\rho)$ and $k >0$ (where $\rho > 0$ is the same index appearing in the statement of Theorem \ref{main}).  The next result can be proved following the same argument as in the proof of \cite[Proposition 3.1]{AscanelliCappielloHolder}. Compared to the latter result, in the present case, the conjugation can be performed for every $\rho' >0$ since the symbol of the operator satisfies \textit{projective} Gevrey estimates. 
	Namely, we have the following result.
	\begin{Prop}\label{secondconjthm}
		Let $p \in \tilde{\Gamma}^{m}_{\theta}(\R^{2})$.
		Then we can write 
		$$
		e^{\Lambda_{\rho',k}}(t,D)\circ p(x,D) \circ e^{-\Lambda_{\rho',k}}(t,D) = op \left( \sum_{\alpha < N} \frac{1}{\alpha!} \partial^{\alpha}_{\xi} e^{\Lambda_{\rho',k}(t,\xi)} D^{\alpha}_{x} p(x,\xi) e^{-\Lambda_{\rho',k}(t,\xi)} \right) + r_N(t,x,D),
		$$
		where $r_{N}$ satisfies: for every $A > 0$ there exists $C_{k,\rho',A} > 0$ such that 
		$$
		|\partial^{\alpha}_{\xi} \partial^{\beta}_{x} r_{N}(t,x,\xi)| \leq |p|_{A} C_{k,\rho',A,N} A^{\alpha+\beta} \langle \xi \rangle^{m-N(1-\frac{1}{\theta})}.
		$$
	\end{Prop}

	\section{The linearized problem} \label{linearized}

	Fixed $u \in \Omega \subset X_{T} := C^{1}([0,T]; H^\infty_\theta(\R))$, where $\Omega$ denotes a bounded set, we now consider the linear Cauchy problem  
	\beqs\label{CPv}
	\left\{\begin{array}{ll}
		P_u(D)v(t,x):=P(t,x,u(t,x),D_t,D_x)v(t,x) = f(t,x), & (t,x)\in [0,T]\times\R,
		\\
		v(0,x)=g(x),
	\end{array}\right.
	\eeqs
	in the unknown $v$. In this section we shall prove the following result.
	
	\begin{Th}\label{iop}
		Under the assumptions of Theorem \ref{main}, given $m\in\R,\ \rho>0$, $\theta\in \left[\theta_0,\frac1{2(1-\sigma)}\right)$, 
		$u\in \Omega \subset C^{1}([0,T];H^{\infty}_{\theta}(\R))$, 
		$f\in C([0,T]; H^m_{\rho;\theta}(\R))$ and $g\in H^m_{\rho;\theta}(\R)$, there exists a unique solution
		$v\in C^1([0,T]; H^{m}_{\rho-\delta;\theta}(\R))$ for every $\delta \in (0,\rho)$ of the Cauchy problem \eqref{CPv}
		and the following energy
		estimate is satisfied:
		\beqs
		\label{enestv}
		\|v(t,\cdot)\|^2_{H^m_{\rho-\delta;\theta}}\leq C_{\Omega,\rho,T}
		\left(\|g\|^2_{H^m_{\rho;\theta}}+
		\int_0^t\|f(\tau,\cdot)\|^2_{H^m_{\rho;\theta}} \,d\tau\right)\quad
		\forall t\in[0,T],
		\eeqs
		for some positive constant $C_{\Omega,\rho,T}$. Moreover, if $f\in C([0,T], H^{\infty}_{\theta}(\R)),\ g\in H^{\infty}_{\theta}(\R)$, then $v\in C^{1}([0,T]; H^{\infty}_{\theta}(\R))$.
	\end{Th}
	
In order to prove the theorem above we shall follow the same method used to prove the well-posedness of the Cauchy problem for linear $3$-evolution equations in $\mathcal{H}^\infty_\theta(\R)$ in \cite{AACgev}. This method is based on making a suitable change of variable in order to transform the Cauchy problem \eqref{CPv} for the operator $P_u(D)$ into an equivalent Cauchy problem which turns out to be well-posed in Sobolev spaces. The transformation we have in mind will be the composition of two transformations both defined by invertible pseudodifferential operators of infinite order. Namely it will be of the form 
	\begin{equation}\label{lalambda}
		Q_{\tilde{\Lambda}, k, \rho'}(t,x,D)=e^{\Lambda_{\rho',k}}(t,D) \circ e^{\tilde\Lambda}(x,D), 
	\end{equation}
	where $\tilde{\Lambda} = \lambda_2 + \lambda_1 \in S^{2(1-\sigma)}_{\mu} (\R^2)$ for some $\mu>1$, and $\Lambda_{\rho',k}(t,\xi)= \rho' \langle \xi \rangle_h^{\frac1\theta}+k(T-t)\langle \xi \rangle_h^{2(1-\sigma)}$ for some $\rho' \in (0,\rho), k >0$ and $h >> 1$ to be chosen later on. Then, by the inverse transformation, we recover the solution $v=Q_{\tilde{\Lambda}, k, \rho'}(t,x,D)^{-1}w$ of \eqref{CPv}, where $w$ stands for the solution of the auxiliary problem. The mapping properties of the transformations $Q_{\tilde{\Lambda}, k, \rho'}(t,x,D)$ and $Q_{\tilde{\Lambda}, k, \rho'}(t,x,D)^{-1}$ will determine the space where the Cauchy problem \eqref{CPv} is well-posed.
	The role of each part of the transformation $Q_{\tilde{\Lambda}, k, \rho'}(t,x,D)$ will be, broadly speaking, the following:
	\begin{itemize}
		\item In the transformation $e^{\tilde{\Lambda}}(x,D)$ the functions $\lambda_1$ and $\lambda_2$ will play two different roles: namely $\lambda_2$ will not change $a_3D_x^3$, but it will change the operator $a_2D_x^2$ into the sum of a positive operator plus a remainder of order $1$ which satisfies the same assumptions as $a_1D_x$, plus an error of order $2(1-\sigma)$ whereas $\lambda_1$ will not change the terms of order 2 and 3, but it will turn the terms of order $1$ into the sum of a positive operator, plus a remainder of order zero, plus an error of order at least $2(1-\sigma)$;
		\item the transformation with $e^{k(T-t)\langle D\rangle_h^{2(1-\sigma)}}$ will not change the terms of order 1, 2 and 3, but it will correct the error of order $2(1-\sigma)$, changing it into the sum of a positive operator plus a remainder of order zero;
		\item Finally, the transformations with $e^{\rho'\langle D\rangle_h^{\frac1\theta}}$ simply moves the setting of the Cauchy problem from Gevrey-Sobolev spaces spaces to standard Sobolev spaces: since $2(1-\sigma) < 1/\theta$ the leading part of $Q_{\tilde{\Lambda},k, \rho'}(t,x,\xi)$ is $e^{\rho' \langle \xi \rangle^{\frac{1}{\theta}}_{h}}$, then the inverse of $Q_{\tilde{\Lambda},k, \rho'}(t,x,D)$ possesses regularizing properties with respect to the spaces $H^m_{\rho; \theta}$, because $\rho'>0$.
	\end{itemize}
	Working step by step, in the next subsection we define the symbol $\tilde\Lambda$ and briefly state its main features, then in Subsection \ref{wpp} we perform the conjugation $Q_{\tilde{\Lambda}, k, \rho'} (iP_u)Q_{\tilde{\Lambda}, k, \rho'}^{-1}$, and finally in Subsection \ref{wpp2} we prove Theorem \ref{iop}.

	\subsection{Change of Variables}
	
	For $M_2, M_1 > 0$ and $h \geq 1$ a large parameter, we define 
	\begin{equation}\label{equation_definition_lambda_2}
		\lambda_2(x, \xi) = M_2 w\left(\frac{\xi}{h}\right) \int_{0}^{x} \langle y \rangle^{-\sigma}  
		\psi\left(\frac{\langle y \rangle}{\langle \xi \rangle^{2}_{h}}\right) dy, \quad (x, \xi) \in \R^2,
	\end{equation}
	\begin{equation}\label{equation_definition_lambda_1}
		\lambda_1(x, \xi) = M_1 w\left(\frac{\xi}{h}\right) \langle \xi \rangle^{-1}_{h} \int_{0}^{x} \langle y \rangle ^{-\frac{\sigma}{2}} \psi\left(\frac{\langle y \rangle}{\langle \xi \rangle^{2}_{h}}\right) dy, \quad (x, \xi) \in \R^2,
	\end{equation}
	where
	$$
	w (\xi) =
	\begin{cases}
		0, \quad\,\,\,\,\, \quad \quad |\xi| \leq 1,\\
		-\textrm{sgn}\, a_3, \quad |\xi| \geq 2,
	\end{cases}
	\quad 
	\psi (y) =
	\begin{cases}
		1, \quad |y| \leq \frac{1}{2}, \\
		0, \quad |y| \geq 1,
	\end{cases}
	$$	
	$|\partial^{\alpha}_{\xi} w(\xi)| \leq C_{w}^{\alpha + 1} \alpha!^{\mu}$, $|\partial^{\beta}_{y} \psi(y)| \leq C_{\psi}^{\beta + 1}\beta!^{\mu}$, with $\mu > 1$.
	
	The functions $\lambda_1$ and $\lambda_2$ have been introduced in \cite{AACgev}. They satisfy peculiar estimates where the powers of the weight functions $\langle \xi \rangle_h$ and $\langle x \rangle$ can be adjusted as needed thanks to the special structure of $\supp \psi$ and $\supp \psi'$. These estimates are contained in the following two lemmas which have been proved in \cite{AACgev}.
	\begin{Lemma}\label{lemma_estimates_lambda_2}
		Let $\lambda_2(x, \xi)$ as in (\ref{equation_definition_lambda_2}). Then the following estimates hold:
		\begin{itemize}
			\item[(i)] $|\partial^{\alpha}_{\xi}\lambda_2(x, \xi)| \leq M_2 C^{\alpha+1}_{\lambda_2} \alpha!^{\mu} \langle \xi \rangle^{-\alpha}_{h} \min\{\langle \xi \rangle^{2(1-\sigma)}_{h}, \langle x \rangle^{1-\sigma} \}$, for $\alpha \geq 0$;
			\item[(ii)] $| \partial^{\alpha}_{\xi} \partial^{\beta}_{x}\lambda_2(x, \xi)| \leq M_2 C^{\alpha+\beta +1}_{\lambda_2} \alpha!^{\mu} \beta!^{\mu} 
			\langle \xi \rangle^{-\alpha}_{h} \langle x \rangle^{-\sigma -\beta+1}$, for $\alpha \geq 0, \beta \geq 1$,
		\end{itemize}
		where $C_{\lambda_2}$ is a constant depending only on $C_{w}, C_{\psi}$ and $\sigma$. 
	\end{Lemma}
	
	\begin{Lemma}\label{lemma_estimates_lambda_1}
		Let $\lambda_1(x,\xi)$ as in (\ref{equation_definition_lambda_1}). Then
		\begin{itemize}
			\item[(i)] $|\partial^{\alpha}_{\xi}\lambda_1(x, \xi)| \leq M_1 C^{\alpha+1}_{\lambda_1} \alpha!^{\mu} \langle \xi \rangle^{-\alpha}_{h} \min \{\langle \xi \rangle^{1-\sigma}_{h},  \langle \xi \rangle^{-1}_{h} \langle x \rangle^{1 - \frac{\sigma}{2}}, \langle x \rangle^{\frac{1}{2} - \frac{\sigma}{2}} \}$, for $\alpha \geq 0$;
			\item[(ii)] $| \partial^{\alpha}_{\xi} \partial^{\beta}_{x}\lambda_1(x, \xi)| \leq M_1 C^{\alpha+\beta +1}_{\lambda_1} \alpha!^{\mu} \beta!^{\mu} 
			\langle \xi \rangle^{-\alpha}_{h} \langle x \rangle^{-\frac{\sigma}{2} -\beta+1} \min \{\langle \xi \rangle_h^{-1}, \langle x \rangle^{-\frac{\sigma}2}\}$, for $\alpha \geq 0, \beta \geq 1$,
		\end{itemize}
		where $C_{\lambda_1}$ is a constant depending only on $C_{w}, C_{\psi}$ and $\sigma$. 
	\end{Lemma}
	
	\begin{Rem}
		Lemmas $\ref{lemma_estimates_lambda_2}$ and $\ref{lemma_estimates_lambda_1}$ imply $\lambda_2, \lambda_1 \in \textbf{\textrm{SG}}^{0, 1-\sigma}_{\mu}(\R^{2})$. Moreover, we also have that $\lambda_1 \in S^{1-\sigma}_{\mu}(\R^{2})$ and $\lambda_2 \in S^{2(1-\sigma)}_{\mu}(\R^{2})$. 
	\end{Rem}
	
	The following result proves the invertibility of the transformation $e^{\tilde\Lambda}(x,D)$ and expresses the inverse  in terms of a composition of  $^R\{e^{-\tilde{\Lambda}}(x,D)\}$ with a Neumann series, see \cite[Lemma 4]{AAC3evolGelfand-Shilov} for the proof. In the statement we shall denote by $\Sigma _\kappa(\R^2)$ the space of all symbols $\tau(x,\xi)$ satisfying for every $A>0, c>0$ an estimate of the form
	$$|\partial_\xi^\alpha \partial_x^\beta  \tau (x,\xi) | \leq C_A A^{\alpha+\beta}(\alpha!\beta!)^{\kappa} e^{-c(\langle x \rangle^{1/k}+\langle \xi \rangle_h^{1/k})},$$ cf. \cite{Pi}. 
	
	\begin{Lemma}\label{lemma_inverse_of_e_power_tilde_Lambda}
		Let $\mu > 1$. For $h \geq 1$ large enough, the operator $e^{\tilde{\Lambda}}(x,D)$ is invertible 
		and its inverse is given by 
		$$
		\{e^{\tilde{\Lambda}}(x,D)\}^{-1} = \hskip2pt ^R  \{e^{-\tilde{\Lambda}}(x,D)\} \circ \sum_{j \geq 0} (-r(x,D))^{j}, 
		$$
		for some $r = \tilde{r} + \bar{r}$, where $\tilde{r} \in \textbf{\textrm{SG}}^{-1,-\sigma}_{\mu}(\R^{2})$, $\bar{r} \in \Sigma_{\kappa}(\R^{2})$ for every 
		$\kappa> 2\mu-1$ and 
		$$
		\tilde{r} - \sum_{1 \leq \gamma \leq N} \frac{1}{\gamma!} \partial^{\gamma}_{\xi}(e^{\tilde{\Lambda}} D^{\gamma}_{x} e^{-\tilde{\Lambda}}) \in\textbf{\textrm{SG}}^{-1-N,-\sigma-\sigma N}_{\mu}(\R^{2}), \quad \forall N \geq 1.
		$$ 
		Moreover, $\sum (-r(x,D))^{j}$ has symbol in $\textbf{\textrm{SG}}^{0,0}_{\mu}(\R^{2}) + \Sigma_{\kappa}(\R^{2})$ for every $\kappa >2\mu-1$. Finally, we have	
		\begin{equation}\label{equation_inverse_of_e_power_tilde_Lambda_in_a_precise_way}
			\{e^{\tilde{\Lambda}}(x,D) \}^{-1} = \hskip2pt ^R \{e^{-\tilde{\Lambda}}(x,D)\} \circ \textrm{op} ( 1 - i\partial_{\xi} \partial_{x} \tilde{\Lambda} - \frac{1}{2}\partial^{2}_{\xi}(\partial^{2}_{x}\tilde{\Lambda} - [\partial_x \tilde{\Lambda}]^{2}) - [\partial_{\xi} \partial_{x} \tilde{\Lambda} ]^{2} + q_{-3} ),
		\end{equation}
		where $q_{-3} \in \textbf{\textrm{SG}}^{-3,-3\sigma}_{\mu}(\R^{2}) + \Sigma_{\kappa}(\R^{2})$.
	\end{Lemma}
	
	\begin{Rem}
		Since we can choose $\mu > 1$ arbitrarily close to $1$, we may assume $2\mu-1 < \theta$. Therefore we can take $\kappa < \theta$ in the above lemma.  
	\end{Rem}

		\subsection{Estimates for the linearized coefficients}
	Before starting to prove Theorem \ref{iop}, we need to state which type of estimates the coefficients of the linearized problem \eqref{CPv} satisfy under the assumptions of Theorem \ref{main}.
	
	Since $\Omega \subset X_T$ is bounded, we have that for any $n \in \N$ there exists $B_n > 0$ such that
	$$
	\sup_{w \in \Omega} \|w\|_{H^{0}_{n;\theta}} \leq B_n.  
	$$  
	On the other hand we can write 
	\begin{align*}
		D^{\alpha}_{x} u(t,x) = \int e^{i \xi x} e^{-\rho \langle \xi \rangle^{\frac{1}{\theta}}} \xi^{\alpha} e^{\rho \langle \xi \rangle^{\frac{1}{\theta}}} \widehat{u}(t,\xi) \dslash\xi.
	\end{align*}
	Since $u(t) \in H^{\infty}_{\theta}$, for any $\rho > 0$ from H\"older inequality we get
	\begin{align*}
		|D^{\alpha}_{x} u(t,x)|^{2} &\leq \int e^{-2\rho \langle \xi \rangle^{\frac{1}{\theta}}} \xi^{2\alpha} d\xi\, \|u(t)\|^{2}_{H^{0}_{\rho;\theta}} \\
		&\leq \left( \frac{2\theta}{\rho} \right)^{2\theta \alpha} \alpha!^{2\theta} \| e^{-\frac{\rho}{2} \langle \cdot \rangle}\|^{2}_{L^{2}} \|u(t)\|^{2}_{H^{0}_{\rho;\theta}}.
	\end{align*}
	The above estimate implies that for any $A > 0$ there is a positive constant $C_{\Omega, A}$ such that
	\begin{equation}\label{eq_projective_estimates_for_u_in_Omega}
		|D^{\alpha}_{x} u(t,x)| \leq C_{\Omega, A} A^{\alpha} \alpha!^{\theta},\quad t\in [0,T], x \in \R, \alpha \in \N_0, 
	\end{equation}
	for every $u \in \Omega$. In particular, we conclude that the values of $w = u(t,x)$ lie in a fixed compact set $K_{\Omega} = K \subset \C( \approx \R^{2})$ for every $u \in \Omega$. We shall fix this compact from now on. Using \eqref{eq_projective_estimates_for_u_in_Omega} and the fact that $a_j \in C([0,T], \Gamma^{\theta_0, \frac{j\sigma}{2}}(\R \times \C)), j=0,1,2$,
	in the next we shall estimate the $x-$derivatives of the maps $x \mapsto a_j(t,x,u(t,x))$. For this we need the Fa\`a di Bruno formula in several variables: let $g = (g_1, \ldots, g_p): \R^{n}\to \R^{p}$, $f:\R^{p} \to \R$ and $\beta \in \N^{n}_{0}$, then 
	\begin{equation} \label{Faa}
		D^{\beta}(f\circ g)(x) = \sum_{\ast} \frac{\beta!}{k_1!\ldots k_{\ell}!} \{D^{k_1+\dots +k_{\ell}}f\}(g(x)) \prod_{j=1}^{\ell} \prod_{i=1}^{p} 
		\left[ \frac{ D^{\delta_j}g_{i}(x) }{ \delta_j! } \right]^{k_{ji}}
	\end{equation}
	where the notation $ \sum\limits_\ast $ means that the sum is taken over all $\ell \in \N$, all sets $\{\delta_1, \ldots, \delta_\ell\}$ of $\ell$ distinct elements of $\N^{n}_{0}-\{0\}$ and all $(k_1, \ldots, k_\ell) \in (\N^{p}_{0}-\{0\})^{\ell}$ such that 
	$
	\sum_{j = 1} |k_j|\delta_j = \beta.
	$
	We also report two useful inequalities: 
	$$
	|k_1+\dots+k_{\ell}|! |\delta_1|!^{|k_1|} \cdots |\delta_\ell|!^{|k_\ell|} \leq |\beta|!
	$$ 
	and
	$$
	\sum \frac{\beta!}{k_1!\ldots k_{\ell}!} \lambda^{|k_1+\dots+k_{\ell}|} \leq C_{\lambda}^{|\beta|+1}, \quad \forall\, \lambda > 0,
	$$
	where $\beta, \ell, (\delta_1, \ldots, \delta_\ell), (k_1, \ldots, k_{\ell})$ are as in formula \eqref{Faa}. For details on Fa\`a di Bruno formula we address the reader to Proposition $4.3$, Corollary $4.5$ and Lemma $4.8$ of \cite{Faadibruno}.
	
	Let now $\beta \in \N_0$, then
	$$
	D^{\beta}_{x}(a_j(t,x,u(t,x))) = \sum_{\ast} \frac{\beta!}{k_1!\ldots k_{\ell}!} \{D^{k_1+\cdots + k_{\ell}}_{(x,w)}a_j\}(t,x,u(t,x)) \prod_{j=1}^{\ell} \prod_{i=1}^{3} 
	\left[ \frac{ D^{\delta_j}_{x}g_{i}(t,x) }{ \delta_j! } \right]^{k_{ji}},
	$$
	where $g_1(t,x) = x$, $g_2(t,x) = Re\, u(t,x)$ and $g_3(t,x) = Im\, u(t,x)$. Applying \eqref{eq_projective_estimates_for_u_in_Omega} and the assumptions on the $a_j$, we get for every $A, B>0$: 
	\begin{align*}
		|D^{\beta}_{x}(a_j(t,x,u(t,x)))| &\leq \sum_\ast \frac{\beta!}{k_1!\ldots k_{\ell}!} C_{K,A} A^{|k_1+\cdots + k_\ell|} |k_1+\cdots+k_\ell|!^{\theta_0} \langle x \rangle^{-\sigma j/2} \\
		&\times \prod_{j=1}^{\ell} \prod_{i=1}^{3} \left[ C_{\Omega, B} B^{\delta_j} \delta_j!^{\theta-1} \right]^{k_{ji}} \\
		&\leq C_{K,A} B^{\beta} \langle x \rangle^{-\sigma j/2} \beta! \\
		&\times \sum_\ast \frac{|k_1+\cdots+k_\ell|!}{k_1!\ldots k_{\ell}!} (C_{\Omega, B} A)^{|k_1+\cdots+k_\ell|} \underbrace{|k_1 + \cdots + k_{\ell}|!^{\theta-1} \prod_{j=1}^{\ell} \delta_j!^{(\theta-1)|k_j|}}_{\leq \beta!^{\theta-1}} \\
		&\leq C_{K,A} B^{\beta} \langle x \rangle^{-\sigma j/2} \beta!^{\theta}\sum_\ast \frac{|k_1+\cdots+k_\ell|!}{k_1!\ldots k_{\ell}!} (C_{\Omega, B} A)^{|k_1+\cdots+k_\ell|}.
	\end{align*}
	Taking $A = C_{\Omega, B}^{-1}$ it follows
	\begin{align*}
		|D^{\beta}_{x}(a_j(t,x,u(t,x)))| &\leq C_1 C_{K,\Omega,B} \{C_1B\}^{\beta}  \beta!^{\theta} \langle x \rangle^{-\sigma j/2}
	\end{align*}
	for some constant $C_1 > 0$ independent of $A$ and $B$ and $C_{K,\Omega,B} > 0$ which in fact depends only on $B$ and $\Omega$. Rescaling the constant $C_1 B$ we obtain the following result.
	
	\begin{Lemma} \label{compositionformula} Under the assumptions of Theorem \ref{main}, let $\Omega \subset X_T$ be a bounded subset. Then for every $B>0$ there exists a constant $C_{\Omega, B} > 0$ such that 
		\begin{equation}\label{eq_estimate_linearized_coefficients}
			|D^{\beta}_{x}(a_j(t,x,u(t,x)))| \leq C_{\Omega,B} B^{\beta}  \beta!^{\theta} \langle x \rangle^{-\sigma j/2}, \quad t \in [0,T], x \in \R, \beta \in \N_0,
		\end{equation}
		for every $u \in \Omega$.
	\end{Lemma} 
	
	\subsection{The conjugation procedure}\label{wpp}
	In the present subsection we perform, step by step, the conjugations needed to obtain the operator $Q_{\tilde{\Lambda}, k,\rho' }(iP_u)Q_{\tilde{\Lambda}, k,\rho' }^{-1}$.
	
	\subsubsection{Conjugation with $e^{\tilde\Lambda}$}
	Now we perform the conjugation  of $iP_u$ by the operator $e^{\tilde\Lambda}(x,D)$, with $\tilde\Lambda(x,\xi) = \lambda_2(x,\xi) + \lambda_1(x,\xi)$. In the next computation, by abuse, we shall denote by $a_3(t,D)$ and $a_j(t,x,u,D)$ for $j=1,2$, the operators $a_3(t)D_x^3$ and $a_j(t,x,u)D_x^j, j=1,2$, respectively, and by $a_1,a_2,a_3$ their symbols, sometimes omitting the dependence on the variables $t,x,u$ and $\xi$.
	
	\begin{itemize}
		\item Conjugation of $ia_3(t,D)$: Since $a_3$ does not depend on $x$, Theorem \ref{conjthmnew} simplifies into (omitting $(t,x,D)$ in the notation)
		\begin{multline*}
			e^{\tilde{\Lambda}}(x,D) \circ ia_3(t,D) \circ  ^{R}\{e^{-\tilde{\Lambda}}(x,D)\} = ia_3(t,D) \\
			+ \textrm{op}\left(\partial_{\xi}\{ia_3 D_x(-\tilde{\Lambda})\} + \frac{1}{2} \partial^{2}_{\xi} \{ia_3 [D^{2}_{x}(-\tilde{\Lambda}) +(D_x\tilde{\Lambda})^{2}]\} + q_{3} + r_{\infty} \right).
		\end{multline*}
		Since $x-$derivatives kill the $\xi-$growth given by the integrals of $\tilde{\Lambda}$, we can conclude that $q_{3}$ has order zero. Composing with the Neumman series we get
		
		\begin{multline*}
			e^{\tilde{\Lambda}}(x,D) ia_3(t,D) \{e^{\tilde{\Lambda}}(x,D)\}^{-1} 
			= \textrm{op}\left(ia_3 - \partial_{\xi}(a_3 \partial_{x}\tilde{\Lambda}) + \frac{i}{2}\partial^{2}_{\xi}[a_3(\partial^{2}_{x}\tilde{\Lambda} - (\partial_{x}\tilde{\Lambda})^2)] + q_3 + r_{\infty}\right) \\ 
			\circ \textrm{op}\left( 1 - i\partial_{\xi} \partial_{x} \tilde{\Lambda} - \frac{1}{2}\partial^{2}_{\xi}(\partial^{2}_{x}\tilde{\Lambda} - [\partial_x \tilde{\Lambda}]^{2}) - [\partial_{\xi} \partial_{x} \tilde{\Lambda} ]^{2} + q_{-3} \right) \\
			= ia_3(t,D) +\textrm{op}\left(- \partial_{\xi}(a_3 \partial_{x}\tilde{\Lambda}) + \frac{i}{2}\partial^{2}_{\xi}\{a_3(\partial^{2}_{x}\tilde{\Lambda} - \{\partial_{x}\tilde{\Lambda}\}^2)\}
			+ a_3\partial_{\xi}\partial_{x}\tilde{\Lambda} - i\partial_{\xi}a_3\partial_{\xi}\partial^{2}_{x}\tilde{\Lambda} \right)
			\\
			+ \textrm{op}\left(i\partial_{\xi}(a_3 \partial_{x}\tilde{\Lambda})\partial_{\xi}\partial_{x}\tilde{\Lambda} - \frac{i}{2}a_3\{\partial^{2}_{\xi}(\partial^{2}_{x}\tilde{\Lambda} + [\partial_x \tilde{\Lambda}]^{2}) + 2[\partial_{\xi} \partial_{x} \tilde{\Lambda} ]^{2}\} + \tilde{r}_0 \right) \\
			= ia_3(t,D)+\textrm{op}\left(- \partial_{\xi}a_3 \partial_{x}\tilde{\Lambda} + \frac{i}{2}\partial^{2}_{\xi}\{a_3[\partial^{2}_{x}\tilde{\Lambda} - 
			(\partial_{x}\tilde{\Lambda})^2]\}
			- i\partial_{\xi}a_3\partial_{\xi}\partial^{2}_{x}\tilde{\Lambda} \right)
			\\
			+ \textrm{op}\left( i\partial_{\xi}(a_3 \partial_{x}\tilde{\Lambda})\partial_{\xi}\partial_{x}\tilde{\Lambda} - \frac{i}{2}a_3\{\partial^{2}_{\xi}(\partial^{2}_{x}\tilde{\Lambda} + [\partial_x \tilde{\Lambda}]^{2}) + 2(\partial_{\xi} \partial_{x} \tilde{\Lambda} )^{2}\} + \tilde{r}_0 \right), \\
		\end{multline*}
		where $\tilde{r}_0 \in C([0,T]; \tilde\Gamma^{0}_{\theta}(\R^{2}))$. From now on we are going to denote by $\tilde r_0$ all remainders of class $C([0,T]; \tilde\Gamma^{0}_{\theta}(\R^{2}))$ satisfying uniform estimates with respect to $u \in \Omega$. Writing $\tilde{\Lambda} = \lambda_2 + \lambda_1$ and noticing that $D_x \lambda_1$ has order $-1$ we get
		\begin{multline*}
			e^{\tilde{\Lambda}}(x,D) ia_3(t,D) \{e^{\tilde{\Lambda}}(x,D)\}^{-1} = ia_3(t,D) \\ +\textrm{op}\left( - \partial_{\xi}a_3 \partial_{x}\lambda_2 - \partial_{\xi}a_3\partial_{x}\lambda_1 + \frac{i}{2}\partial^{2}_{\xi}\{a_3(\partial^{2}_{x}\lambda_2 - \{\partial_{x}\lambda_2\}^2)\}
			- i\partial_{\xi}a_3\partial_{\xi}\partial^{2}_{x}\lambda_2 \right) 
			\\
			+ \textrm{op}\left( i\partial_{\xi}(a_3 \partial_{x}\lambda_2)\partial_{\xi}\partial_{x}\lambda_2 - \frac{i}{2}a_3\{\partial^{2}_{\xi}(\partial^{2}_{x}\lambda_2 + [\partial_x \lambda_2]^{2}) + 2[\partial_{\xi} \partial_{x} \lambda_2 ]^{2}\} + \tilde{r}_0 \right).
		\end{multline*}
		For simplicity we write in short 
		\begin{align*}
			d_1(t,x,\xi) &= \frac{1}{2}\partial^{2}_{\xi}\{a_3(\partial^{2}_{x}\lambda_2 - \{\partial_{x}\lambda_2\}^2)\}
			- \partial_{\xi}a_3\partial_{\xi}\partial^{2}_{x}\lambda_2 \\
			&+ \partial_{\xi}(a_3 \partial_{x}\lambda_2)\partial_{\xi}\partial_{x}\lambda_2 - \frac{1}{2}a_3\{\partial^{2}_{\xi}(\partial^{2}_{x}\lambda_2 + [\partial_x \lambda_2]^{2}) + 2[\partial_{\xi} \partial_{x} \lambda_2 ]^{2}\}.
		\end{align*}
		Hence 
		$$
		e^{\tilde{\Lambda}}(x,D) ia_3(t,D) \{e^{\tilde{\Lambda}}(x,D)\}^{-1} = ia_3 (t,D) +\textrm{op}\left(- \partial_{\xi}a_3 \partial_{x}\lambda_2 - \partial_{\xi}a_3\partial_{x}\lambda_1 + id_1 + \tilde{r}_0 \right).
		$$
		Notice that $d_1$ is a real valued symbol of order $1$ which does not depend on $\lambda_1$. Namely, we have the following estimates: for every $A > 0$ there exists $C_{\tilde\Lambda, A} > 0$ such that 
		\begin{equation} \label{stimad1}
		|\partial^{\alpha}_{\xi} \partial^{\beta}_{x} d_1(t,x,\xi)| \leq C_{\lambda_2, A} A^{\alpha+\beta} \alpha!^{\theta} \beta!^{\theta} \langle \xi \rangle^{1 - \alpha}_{h} 
		\langle x \rangle^{-\sigma}.
		\end{equation}

		\item Conjugation of $ia_2(t,x,u, D)$: for $N \in \N$ such that $2-N(2\sigma-1) \leq 0$, Theorem \ref{conjthmnew} and \eqref{eq_estimate_linearized_coefficients}  give 
		\begin{multline*}
			e^{\tilde{\Lambda}}(x,D) \circ ia_2(t,x,u,D) \circ\, ^{R}\{e^{-\tilde{\Lambda}}(x,D)\} = ia_2 (t,x,u,D) \\ + 
			\text{op}\underbrace{ \left( \sum_{1 \leq \alpha + \beta < N} \frac{1}{\alpha! \beta!} \partial^{\alpha}_{\xi} \{ \partial^{\beta}_{\xi} e^{\tilde{\Lambda}} D^{\beta}_{x} ia_2  D^{\alpha}_{x} e^{-\tilde{\Lambda}}\} \right)}_{=:(ia_2)_{N}} + \tilde{r}_0.	
		\end{multline*}
		Composing with the Neumann series and using that $\partial_x\lambda_1$ has order $-1$ we get
		\begin{multline*}
			e^{\tilde{\Lambda}} (x,D)\circ ia_2(t,x,u,D) \circ \{e^{\tilde{\Lambda}}(x,D)\}^{-1} =\textrm{op} (ia_2 + (ia_2)_{N} + \tilde{r}_0 + \tilde{r})\circ \textrm{op}(1 - i\partial_\xi\partial_x \lambda_2 + q_{-2}) \\
			= ia_2(t,x,u,D) + \textrm{op} ( (ia_2)_{N} + a_2 \circ \partial_\xi\partial_x\lambda_2 - i(ia_2)_{N} \circ\partial_\xi\partial_x \lambda_2 +\tilde{r}_0) \\
			= ia_2(t,x,u,D) + \textrm{op} ( \underbrace{(ia_2)_{N} - i(ia_2)_{N} \partial_\xi\partial_x \lambda_2}_{=: (ia_2)_{\tilde{\Lambda}} } + a_2 \partial_\xi\partial_x\lambda_2 + \tilde{r}_0 ).
		\end{multline*}
		Moreover, in view of \eqref{eq_estimate_linearized_coefficients}, we have the following estimates: for every $A > 0$ there exists $C_{\tilde\Lambda, \Omega, A} > 0$ such that
		\begin{equation}\label{stimaa2lambda}
		|\partial^{\alpha}_{\xi} \partial^{\beta}_{x} (ia_2)_{\tilde{\Lambda}}(t,x, u, \xi)| \leq  C_{\tilde\Lambda, \Omega, A} A^{\alpha+\beta} \alpha!^{\theta} \beta!^{\theta} \langle \xi \rangle^{2 - (2\sigma - 1) - \alpha }_{h} \langle x\rangle^{-\sigma},
		\end{equation}
		for every $u \in \Omega$.

		\item Conjugation of $ia_1(t,x,u, D):$ working as in the previous conjugation, we get
		\begin{multline*}
			e^{\tilde{\Lambda}}(x,D) \circ (ia_1)(t,x,u,D) \circ \{e^{\tilde{\Lambda}}(x,D)\}^{-1} = \textrm{op}(ia_1 + (ia_1)_{\tilde{\Lambda}} + r_1) \circ \sum_{j\geq 0}(-r)^{j} \\  = ia_1(t,x,u,D) + \textrm{op}((ia_1)_{\tilde{\Lambda}} + \tilde{r}_{0}), 
		\end{multline*}
		where we have the following estimates: for every $A > 0$ there exists $C_{\tilde\Lambda, \Omega, A} > 0$ such that for every $u \in \Omega$: 
		\begin{equation} \label{stimaa1lambda}
		|\partial^{\alpha}_{\xi} \partial^{\beta}_{x} (ia_1)_{\tilde{\Lambda}}(t,x, u,\xi)| \leq C_{\tilde\Lambda, \Omega, A} A^{\alpha+\beta} \alpha!^{\theta} \beta!^{\theta} \langle \xi \rangle^{2(1-\sigma) - \alpha }_{h} \langle x\rangle^{-\sigma/2}.
		\end{equation}
		
		\item Conjugation of $ia_0(t,x, u)$: $e^{\tilde{\Lambda}}(x,D) \circ (ia_0)(t,x,u) \circ \{e^{\tilde{\Lambda}}(x,D)\}^{-1} = \tilde{r}_{0}$.
		
	\end{itemize}
	
	Gathering all the previous computations we get (omitting $(t,x, u,D)$ in the notation)
	\begin{align*}
		e^{\tilde{\Lambda}}(x,D) &(iP_u)\{e^{\tilde{\Lambda}}(x,D)\}^{-1} = \partial_{t} + ia_3(t,D) +\textrm{op}\left(- \partial_{\xi}a_3\partial_x\lambda_2 - \partial_{\xi}a_3 \partial_{x}\lambda_1 + id_{1}\right) \\
		&+ ia_2(t,x,u,D) + \textrm{op}((ia_2)_{\tilde{\Lambda}} + a_2\partial_{\xi}\partial_{x}\lambda_2) + ia_1(t,x,u,D) + \textrm{op}((ia_1)_{\tilde{\Lambda}} + \tilde{r}_0).
	\end{align*}
	where $d_1, (ia_2)_{\tilde{\Lambda}}$ and $(ia_1)_{\tilde{\Lambda}}$ 
	satisfy the estimates \eqref{stimad1}, \eqref{stimaa2lambda}, \eqref{stimaa1lambda} 
	for every $u \in \Omega$. 
%
	

	\subsubsection{Conjugation by $e^{\Lambda_{\rho',k}}(t,D),$ with $\Lambda_{\rho',k}(t,\xi) = \rho'\langle \xi \rangle^{\frac{1}{\theta}}_{h} + k(T-t)\langle \xi \rangle^{2(1-\sigma)}_{h}$}

	\begin{itemize}
		\item Conjugation of $ \partial_t$: $e^{\Lambda_{\rho',k}}(t,D) \circ \partial_{t} \circ e^{-\Lambda_{\rho',k}}(t,D) = \partial_{t} + k\langle D\rangle^{2(1-\sigma)}_{h}$.
		
		\item Conjugation of $ia_3(t,D)$: since $a_3$ does not depend of $x$, we simply have 
		$$
		e^{\Lambda_{\rho',k}}(t,D) \circ ia_3 (t,D) \circ e^{-\Lambda_{\rho',k}}(t,D) = ia_3(t,D).
		$$
		
		\item Conjugation of $\text{op}\{ia_2 - \partial_{\xi}a_3\partial_{x}\lambda_2\}$:
		\begin{align*}
			e^{\Lambda_{\rho',k}}(t,D)\circ \text{op}(ia_2 - \partial_{\xi}a_3\partial_{x}\lambda_2 ) &\circ e^{-\Lambda_{\rho',k}}(t,D)= 
			ia_2(t,x,u,D) \\ 
			&- \text{op}(\partial_{\xi}a_3\partial_{x}\lambda_2) + (b_{2,\rho',k} + \tilde{r}_0)(t,x,u,D)
		\end{align*}
		where $b_{2,\rho',k}$ satisfies: for any $A > 0$ there exists $C_{\lambda_2,\Omega, \rho', k, A} > 0$ such that (for every $u \in \Omega$)
		\begin{equation}\label{equation_remainde_conj_level_2_by_k_D}
			|\partial^{\alpha}_{\xi} \partial^{\beta}_{x} b_{2,\rho',k}(t,x,u,\xi)| \leq C_{\lambda_2,\Omega, \rho', k, A} A^{\alpha+\beta} (\alpha!\beta!)^{\theta} \langle \xi \rangle^{2-(1-\frac{1}{\theta}) - \alpha}_{h} \langle x \rangle^{-\sigma}.
		\end{equation}
		
		\item Conjugation of $(ia_2)_{\tilde{\Lambda}} (t,x,u,D)$: 
		$$
		e^{\Lambda_{\rho',k}}(t,D)\circ (ia_2)_{\tilde{\Lambda}}(t,x,u,D) \circ e^{-\Lambda_{\rho',k}}(t,D)= \{(ia_2)_{\rho', k, \tilde{\Lambda}} + \tilde{r}_0\}(t,x,u,D),
		$$
		where $(ia_2)_{\rho', k, \tilde{\Lambda}}$  satisfies: for any $A > 0$ there exists $C_{\tilde{\Lambda},\Omega, \rho', k, A} > 0$ such that (for every $u \in \Omega$)
		\begin{equation}\label{equation_remainder_conj_ia_2_Lambda_tilde_by_k_D}
			|\partial^{\alpha}_{\xi} \partial^{\beta}_{x} (ia_2)_{\rho',k,\tilde{\Lambda}}(t,x,u,\xi)| \leq C_{\tilde{\Lambda},\Omega, \rho', k, A} A^{\alpha+\beta} (\alpha!\beta!)^{\theta} \langle \xi \rangle^{2-(2\sigma - 1) - \alpha}_{h} \langle x \rangle^{-\sigma}.
		\end{equation}
		
		\item Conjugation of $\text{op}\{ia_1  - \partial_{\xi}a_3 \partial_{x}\lambda_1 + id_{1} + a_2\partial_{\xi}\partial_{x}\lambda_2\}$: we have 
		\begin{align*}
			e^{\Lambda_{\rho',k}}(t,D)&\circ \text{op}(ia_1  - \partial_{\xi}a_3 \partial_{x}\lambda_1 + id_{1} + a_2\partial_{\xi}\partial_{x}\lambda_2) \circ e^{-\Lambda_{\rho',k}}(t,D) \\ 
			&= \text{op}(ia_1  - \partial_{\xi}a_3 \partial_{x}\lambda_1 + id_{1} + a_2\partial_{\xi}\partial_{x}\lambda_2+ b_{1,\rho',k} + \tilde{r}_0),
		\end{align*}
		where $b_{1,\rho', k}$  satisfies: for any $A > 0$ there exists $C_{\tilde{\Lambda},\Omega, \rho', k, A} > 0$ such that (for every $u \in \Omega$)
		\begin{equation}\label{equation_remainder_conj_level_1_Lambda_tilde_by_k_D}
			|\partial^{\alpha}_{\xi} \partial^{\beta}_{x} b_{1,\rho',k}(t,x,u,\xi)| \leq C_{\tilde{\Lambda},\Omega, \rho', k, A} A^{\alpha+\beta} (\alpha!\beta!)^{\theta} \langle \xi \rangle^{1 -(1-\frac{1}{\theta}) - \alpha}_{h} \langle x \rangle^{-\sigma/2}.
		\end{equation}
		
		\item Conjugation of $(ia_1)_{\tilde{\Lambda}}(t,x,u,D)$: 
		$$
		e^{\Lambda_{\rho',k}}(t,D) \circ (ia_1)_{\tilde{\Lambda}}(t,x,u,D) \circ e^{-\Lambda_{\rho',k}}(t,D)= \{(ia_1)_{\rho',k, \tilde{\Lambda}} + \tilde{r}_0\}(t,x,u, D),
		$$
		where $(ia_1)_{\rho', k,\tilde{\Lambda}}$  satisfies: for any $A > 0$ there exists $C_{\tilde{\Lambda},\Omega, \rho', k, A} > 0$ such that (for every $u \in \Omega$)
		\begin{equation}\label{equation_remainder_conj_level_less_1_Lambda_tilde_by_k_D}
			|\partial^{\alpha}_{\xi} \partial^{\beta}_{x} (ia_1)_{\rho', k,\tilde{\Lambda}}(t,x,u,\xi)| \leq C_{\tilde{\Lambda},\Omega, \rho', k, A} A^{\alpha+\beta} (\alpha!\beta!)^{\theta} \langle \xi \rangle^{2(1-\sigma) - \alpha}_{h} \langle x \rangle^{-\sigma/2}.
		\end{equation}
		
	\end{itemize}
	
	Finally, gathering all the previous computations we obtain the following expression for the conjugated operator (provided that the parameter $h$ is sufficiently large)
	\begin{align}\label{coop}
		Q_{\tilde\Lambda, k,\rho'} (iP_u) Q_{\tilde\Lambda, k,\rho'} ^{-1} &= \partial_{t} + k\langle D\rangle^{2(1-\sigma)}_{h} + ia_3(t,D) \\ \nonumber
		&+ \text{op}(ia_2  - \partial_{\xi}a_3\partial_x\lambda_2 + b_{2,\rho', k} + (ia_2)_{\rho',k,\tilde{\Lambda}}) \\ \nonumber
		&+ \text{op}(ia_1 - \partial_{\xi}a_3 \partial_{x}\lambda_1 + id_{1}  + a_2\partial_{\xi}\partial_{x}\lambda_2+ b_{1,\rho',k}  + (ia_1)_{\rho', k, \tilde{\Lambda}} ) \\ \nonumber
		&+ \tilde{r}_0(t,x,u,D), \nonumber
	\end{align}
	where $b_{2,\rho',k}$ satisfies \eqref{equation_remainde_conj_level_2_by_k_D}, $(ia_2)_{\rho',k, \tilde{\Lambda}}$ satisfies \eqref{equation_remainder_conj_ia_2_Lambda_tilde_by_k_D}, $b_{1,\rho',k}$ satisfies \eqref{equation_remainder_conj_level_1_Lambda_tilde_by_k_D}, $(ia_1)_{\rho',k, \tilde{\Lambda}}$ satisfies \eqref{equation_remainder_conj_level_less_1_Lambda_tilde_by_k_D}, $\tilde{r}_0$ is a projective symbol of order zero satifying uniform estimates with respect to $u \in \Omega.$


\subsection{Proof of Theorem \ref{iop}}\label{wpp2}
This Subsection is devoted to the proof of Theorem \ref{iop}. First of all we need some estimates from below for the terms appearing in operator \eqref{coop} in order to apply to these terms  Fefferman-Phong and sharp G{\aa}rding inequalities.
Let us start with the terms $\partial_\xi a_3(t,\xi)\partial_x\lambda_j(x,\xi)= 3a_3 (t)\xi^2 \partial_x\lambda_j(x,\xi)$, $j=1,2$. 

For $|\xi| > 2h$, by \eqref{equation_definition_lambda_2} and \eqref{equation_definition_lambda_1}  we have
\begin{align*}
	- \partial_\xi a_3\partial_x\lambda_2(x,\xi) &=  3M_2 |a_3(t)| \xi^2 \langle x \rangle^{-\sigma}  
	\psi\left(\frac{\langle x \rangle}{\langle \xi \rangle^{2}_{h}}\right) \\
	&=  3M_2 |a_3(t)| \xi^2 \langle x \rangle^{-\sigma}  - 3M_2 |a_3(t)| \xi^2 \langle x \rangle^{-\sigma} \left[  1 - \psi\left(\frac{\langle x \rangle}{\langle \xi \rangle^{2}_{h}}\right) \right],
\end{align*}
\begin{align*}
	- \partial_\xi a_3\partial_x\lambda_1 (x,\xi) &= 3 M_1 |a_3(t)| \xi^2 \langle \xi \rangle^{-1}_{h} \langle x \rangle^{-\frac{\sigma}{2}}  
	\psi\left(\frac{\langle x \rangle}{\langle \xi \rangle^{2}_{h}}\right) \\ 
	&=  3M_1 |a_3(t)| \xi^2 \langle \xi \rangle^{-1}_{h} \langle x \rangle^{-\frac{\sigma}{2}}  - 3 M_1|a_3(t)| \xi^2  \langle \xi \rangle^{-1}_{h} \langle x \rangle^{-\frac{\sigma}{2}} \left[  1 - \psi\left(\frac{\langle x \rangle}{\langle \xi \rangle^{2}_{h}}\right) \right].
\end{align*}
Since $\langle x \rangle \geq \frac12 \langle \xi \rangle^{2}_{h}$ on the support of $(1-\psi)(\langle x \rangle \langle \xi \rangle^{-2}_{h})$, we have
$$
- 3M_2|a_3(t)| \xi^2 \langle x \rangle^{-\sigma} \left[  1 - \psi\left(\frac{\langle x \rangle}{\langle \xi \rangle^{2}_{h}}\right) \right] \geq - 2^{\sigma} 3C' M_2 \langle \xi \rangle^{2(1-\sigma)}_{h},
$$ 
and
$$
- 3M_1|a_3(t)|\xi^2 \langle \xi \rangle^{-1}_{h} \langle x \rangle^{-\frac{\sigma}{2}} \left[  1 - \psi\left(\frac{\langle x \rangle}{\langle \xi \rangle^{2}_{h}}\right) \right] \geq - 2^{\frac{\sigma}{2}}3C'M_1 \langle \xi \rangle^{1-\sigma}_{h},
$$
where $C'= \sup_{t \in [0,T]}|a_3(t)|$. 
In this way we may write ($|\xi| > 2h$)
$$Q_{\tilde{\Lambda},k,\rho'} \circ (iP_u)  \circ Q_{\tilde{\Lambda},k,\rho'}^{-1} = \partial_{t} + ia_3(t)D^3_x + \tilde{a}_2(t,x,u,D) + \tilde{a}_1(t,x,u,D) + \tilde{a}_{2(1-\sigma)}(t,x,D) + r_0(t,x,u,D),
$$
where $r_0$ is an operator of order $0$ and 
\beqsn
Re\, \tilde{a}_2 &=& -Im\, a_2 +3 M_2|a_3(t)| \xi^2  \langle x \rangle^{-\sigma} + Re\, b_{2,\rho',k} + Re\, (ia_2)_{\rho',k,\tilde{\Lambda}} ,
\\
Im\, \tilde{a}_2 &=& Re \, a_2 + Im\, b_{2,\rho',k} + Im\, (ia_2)_{\rho', k, \tilde{\Lambda}} ,
\\
Re\, \tilde{a}_1 &=& -Im\, a_1 + 3|a_3(t)|\xi^2 M_1 \langle \xi \rangle^{-1}_{h} \langle x \rangle^{-\frac{\sigma}{2}}  + Re\, a_2\partial_{\xi}\partial_{x}\lambda_2 
+ Re\, b_{1,\rho',k}  + Re\, (ia_1)_{\rho',k, \tilde{\Lambda}},
\\
\tilde{a}_{2(1-\sigma)} &=& k\langle \xi \rangle^{2(1-\sigma)}_{h}  
-3 |a_3(t)|\xi^2 M_2 \langle x \rangle^{-\sigma} \left[  1 - \psi\left(\frac{\langle x \rangle}{\langle \xi \rangle^{2}_{h}}\right) \right] 
- 3|a_3(t)|\xi^2 M_1 \langle \xi \rangle^{-1}_{h} \langle x \rangle^{-\frac{\sigma}{2}} \left[  1 - \psi\left(\frac{\langle x \rangle}{\langle \xi \rangle^{2}_{h}}\right) \right]. 
\eeqsn

Now we decompose $iIm\, \tilde{a}_2$ into its Hermitian and anti-Hermitian part:
$$
i Im\, \tilde a_2= \frac{i Im\, \tilde a_2+(i Im\, \tilde a_2)^*}{2}+\frac{i Im\, \tilde a_2-(i Im\, \tilde a_2)^*}{2}= H_{Im\,\tilde{a}_2} + A_{Im\,\tilde{a}_2};
$$
we have that $2Re\, \langle A_{Im\,\tilde{a}_2} u, u \rangle = 0$, while $H_{Im\,\tilde{a}_2}$ has symbol 
$$
\sum_{\alpha\geq 1}\frac{i}{2\alpha!}\partial_\xi^\alpha D_x^\alpha Im\, \tilde a_2 = 
\underbrace{\sum_{\alpha\geq 1}\frac{i}{2\alpha!}\partial_\xi^\alpha D_x^\alpha Re\, a_2}_{=:c(t,x,u,\xi)} + 
\underbrace{\sum_{\alpha\geq 1}\frac{i}{2\alpha!}\partial_\xi^\alpha D_x^\alpha \{ Im\, b_{2,\rho',k} + Im\, (ia_2)_{\rho', k, \tilde{\Lambda}}  \}}_{=:e(t,x,u,\xi)}.
$$
The hypothesis on $a_2$ implies 
$$
|\partial^{\alpha}_{\xi}\partial^{\beta}_{x} c(t,x,u,\xi)| \leq C_{\Omega,A} A^{\alpha+\beta} (\alpha!\beta!)^{\theta} \langle \xi \rangle_{h}^{1-\alpha} \langle x \rangle^{-\sigma},
$$
whereas from \eqref{equation_remainde_conj_level_2_by_k_D}, \eqref{equation_remainder_conj_ia_2_Lambda_tilde_by_k_D} and using that $2(1-\sigma) \leq \frac{1}{\theta}$ we obtain 
$$
|\partial^{\alpha}_{\xi}\partial^{\beta}_{x} e(t,x,u,\xi)| \leq C_{\tilde{\Lambda},\Omega, \rho', k, A} A^{\alpha+\beta} (\alpha!\beta!)^{\alpha+\beta} \langle \xi \rangle^{\frac{1}{\theta}}_{h} \langle x \rangle^{-\sigma}. 
$$
We are ready to get the desired estimates from below.  Using the above decomposition we get
\begin{align*}
	e^{\Lambda} \circ (iP_u)  \circ \{e^{\Lambda}\}^{-1} = \partial_{t} &+ ia_3(t)D_x^3 + Re\,\tilde{a}_2(t,x,u,D) + A_{Im\,\tilde{a}_2}(t,x,u,D) \\ 
	&+ (\tilde{a}_1+c+e)(t,x,u,D) + \tilde{a}_{2(1-\sigma)}(t,x,D) + r_0(t,x,u,D).
\end{align*}
Note that $\langle \xi \rangle^{2}_{h} \leq 2 \xi^{2}$ provided that $|\xi| > 2h$. In the next we shall fix $A = 1$ in the estimates and we shall omit the dependence on $A$ in the constants. Estimating the terms of order $2$ we get
\beqs \label{lbestA_2}
Re\, \tilde{a}_2 \geq \left(M_2 \frac{3C_{a_3}}{2} - C_{\Omega} - C_{\lambda_2,\Omega,\rho',k} h^{-(1-\frac{1}{\theta})} - C_{\tilde{\Lambda},\Omega,\rho',k}h^{-(2\sigma-1)} \right)\langle \xi \rangle^{2}_{h} \langle x \rangle^{-\sigma}, \nonumber
\eeqs
where $C_{a_3}$ is the constant appearing in the statement of Theorem \ref{main}.
For the terms of order $1$ we obtain
\beqs \label{lbestA_1}
Re\, (\tilde{a}_1+c+e) \geq  \left(M_1 \frac{3C_{a_3}}{2} - C_{\Omega} - C_{\Omega, \lambda_2} - C_{\tilde{\Lambda},\Omega,\rho',k} h^{-(1-\frac{1}{\theta})} -
C_{\tilde{\Lambda},\Omega,\rho',k}h^{-(2\sigma-1)}\right) \langle \xi \rangle_{h} \langle x \rangle^{-\frac{\sigma}{2}}.  \nonumber
\eeqs
Finally, for the terms of order $\leq 2(1-\sigma)$ we have
\begin{align} \label{lbestA_theta}
	Re\, \tilde{a}_{2(1-\sigma)} &\geq k\langle \xi \rangle^{2(1-\sigma)}_{h}  - 2^{\sigma}3C_{a_3} M_2 \langle \xi \rangle^{2(1-\sigma)}_{h}  - 2^{\frac{\sigma}{2}}3C_{a_3}M_1 \langle \xi \rangle^{1-\sigma}_{h} \nonumber 
	\\
	&\geq \left( k - 2^{\sigma}3C_{a_3} M_2  - 2^{\frac{\sigma}{2}}3C_{a_3}M_1 h^{-(1-\sigma)} \right)\langle \xi \rangle^{2(1-\sigma)}_{h}.
\end{align}

From the previous lower bound estimates we obtain the following proposition.   

\begin{Prop}\label{proposition_Hm_well_posedness_for_the_conjugated_problem}
	There exist constants $M_2 , M_1,k>0$ and $h_{0}=h_0(k,M_2,M_1,T,\Omega,\rho')$ $> 0$ such that for every $h \geq h_0$ the Cauchy problem associated with the conjugated operator \eqref{coop} is well-posed in $H^{m}(\R)$. More precisely, for any Cauchy data $\tilde f \in C([0,T];H^{m}(\R))$ and $\tilde g \in H^{m}(\R)$, there exists a unique solution $w \in C([0,T];H^{m}(\R)) \cap C^{1}([0,T]; H^{m-3}(\R))$ such that the following energy estimate holds: there exists a constant $C_{\Omega, \rho', T} > 0$ depending on $\Omega$, $\rho' > 0$ and $T > 0$ such that
	$$
	\|w(t)\|^{2}_{H^{m}} \leq C_{\Omega,\rho',T} \left( \| \tilde g \|^{2}_{H^{m}}  + \int_{0}^{t} \| \tilde f(\tau)\|^{2}_{H^{m}} d\tau \right), \quad t \in [0,T]
	$$
	for every $u \in \Omega$.
\end{Prop}
\begin{proof} 
	Firts we take $M_2 > 0$ large in order to get 
	\begin{equation}\label{firstlbestimate}
		M_2 \frac{3C_{a_3}}{2} - C_{\Omega} > 0,
	\end{equation}
	then we set $M_1 = M_1(M_2) > 0$ in such way that 
	\begin{equation} \label{secondlbestimate}
		M_1 \frac{3C_{a_3}}{2} - C_{\Omega} - C_{\Omega, \lambda_2} > 0.
	\end{equation}
	Thereafter we choose $k = k(M_2) > 0$ such that
	\begin{equation}\label{sck}
		k - 2^{\sigma}3C_{a_3} M_2 > 0.
	\end{equation}
	Now, making the parameter $h_{0}$ large enough, we obtain 
	$$
	M_2 \frac{3C_{a_3}}{2} - C_{\Omega} - C_{\lambda_2,\Omega,\rho',k} h^{-(1-\frac{1}{\theta})} - C_{\tilde{\Lambda},\Omega,\rho',k}h^{-(2\sigma-1)} \geq 0,
	$$
	$$
	M_1 \frac{3C_{a_3}}{2} - C_{\Omega} - C_{\Omega, \lambda_2} - C_{\tilde{\Lambda},\Omega,\rho',k} h^{-(1-\frac{1}{\theta})} -
	C_{\tilde{\Lambda},\Omega,\rho',k}h^{-(2\sigma-1)}  \geq 0, 
	$$
	$$
	k - 2^{\sigma}3C_{a_3} M_2  - 2^{\frac{\sigma}{2}}3C_{a_3}M_1 h^{-(1-\sigma)} \geq 0.
	$$
	With these choices $Re\, \tilde{a}_2(t,x,u,\xi), Re\, (\tilde{a}_1+c+e)(t,x,u,\xi), Re\, \tilde{a}_{2(1-\sigma)}(t,x,\xi)$ are non-negative for large $|\xi|$. Applying the Fefferman-Phong inequality, cf. \cite{FP}, to $Re\, \tilde a_2$ we have 
	$$
	Re \langle Re\, \tilde{a}_2(t,x,u,D) w,w \rangle_{L^{2}} \geq -C \|w \|^2_{L^2}, \quad w \in \mathscr{S}(\R).
	$$ 	
	By the sharp G{\aa}rding inequality, cf. \cite[Theorem 4.4]{KG}, we also obtain that
	$$
	Re \langle (\tilde{a}_1+c+e)(t,x,u,D) w,w \rangle_{L^{2}} \geq -C \|w\|^2_{L^2}, \quad w \in \mathscr{S}(\R)
	$$ 
	and 
	$$
	Re \langle \tilde{a}_2(1-\sigma)(t,x,D) w,w \rangle_{L^{2}} \geq -C \|w \|^2_{L^2}, \quad w \in \mathscr{S}(\R).
	$$ 
	The constant $C > 0$ that we just write in the above inner product estimates depends on a finite number of the symbols seminorms, in this way we have that $C$ depends on $\Omega, \rho', T, \tilde{\Lambda}$ and $k$. As a consequence we get the energy estimate 
	$$
	\dfrac{d}{dt}\| w(t) \|^{2}_{L^{2}} \leq C_{\Omega,\rho', T} (\| w(t) \|^{2}_{L^{2}} + \| (iP)_{\Lambda} w(t) \|^{2}_{L^{2}}),
	$$ 
	which gives us the well-posedness in $H^{m}(\R)$.
\end{proof}

\begin{Rem}
We underline that the assumption $|a_3(t)|\geq C_{a_3}>0, \forall t\in[0,T]$is crucial in the choice of $M_2,M_1$. If $a_3$ may vanish for some $t\in [0,T]$, then some Levi type conditions are needed on $a_2,a_1$ to let the choice of $M_2,M_1$ work, see \cite{ABZ}. 
\end{Rem}

Now we are ready to prove Theorem \ref{iop}.
\\
{\it Proof of Theorem \ref{iop}}. Given $m\in\R$ and $\theta>1,$ take $f \in C([0,T], H^{m}_{\rho; \theta}(\R))$ and  $g \in H^{m}_{\rho; \theta}(\R)$ for some $\rho > 0$. Let $M_2, M_1, k, h_0 > 0$ so that Proposition \ref{proposition_Hm_well_posedness_for_the_conjugated_problem} holds. Since 
$\tilde{\Lambda}$ and $k(T-t)\langle \cdot\rangle_h^{2(1-\sigma)}$ have order $2(1-\sigma)<\frac{1}{\theta}$, we have by Proposition \ref{contgev} that 
\beqsn
f_{\tilde\Lambda,k,\rho'}:=Q_{\tilde\Lambda,k,\rho'}(t,x,D)f \in C([0,T]; H^{m}(\R))
\\ g_{\tilde\Lambda,k,\rho'}:= Q_{\tilde\Lambda,k,\rho'}(0,x,D) g \in H^{m}(\R),
\eeqsn
provided that $\rho' < \rho$. Proposition \ref{proposition_Hm_well_posedness_for_the_conjugated_problem} ensures that the Cauchy problem associated with the operator in \eqref{coop}, call it $P_{\tilde\Lambda,k,\rho',u}$, is well posed in Sobolev spaces $H^{m}(\R)$. Hence, there exists a unique $w \in C([0,T]; H^{m}(\R))$ satisfying 
$$
\begin{cases}
	P_{\tilde\Lambda,k,\rho',u} w(t,x) = f_{\tilde\Lambda,k,\rho'}(t,x), \\
	w(0,x) = g_{\tilde\Lambda,k,\rho'}(x),
\end{cases}
$$ 
and
\begin{equation}\label{donkey_kong_1}
	\|w(t)\|^{2}_{H^{m}} \leq C_{\Omega, \rho', T} \left( \| g_{\tilde\Lambda,k,\rho'} \|^{2}_{H^{m}}  + \int_{0}^{t} \|f_{\tilde\Lambda,k,\rho'}(\tau)\|^{2}_{H^{m}} d\tau \right), \quad t \in [0,T].
\end{equation}

Setting $v = \{Q_{\tilde\Lambda,k,\rho'}(t,x,D)\}^{-1} w$ we obtain a solution for the original problem \eqref{CPv}.
Let us now study which space the solution $v$ belongs to. We have 
\beqsn
v(t,x) &=&  \{Q_{\tilde\Lambda,k,\rho'}(t,x,D)\}^{-1} w(t,x)
\\
&= &  ^{R}\{e^{-\tilde{\Lambda}}\}\!(x,D) \sum_{j} (-r(x,D))^{j} e^{-k(T-t)\langle D \rangle^{2(1-\sigma)}_{h}}e^{-\rho' \langle D \rangle^{\frac{1}{\theta}}_{h}} w(t,x),\quad w \in H^{m}(\R).
\eeqsn
Since $e^{-\rho' \langle D \rangle^{\frac{1}{\theta}}_{h}} w=:v_1 \in H^{m}_{\rho'; \theta} (\R),$ we get
\beqsn
v(t,x)= \,^{R}\{e^{-\tilde{\Lambda}}(x,D)\} \sum_{j} (-r(x,D))^{j} e^{-k(T-t)\langle D \rangle^{2(1-\sigma)}_{h}} v_1,\quad v_1\in H^{m}_{\rho'; \theta} (\R),
\eeqsn 
but for every $\delta_1>0$ $e^{-k(T-t)\langle D \rangle^{2(1-\sigma)}_{h}} v_1=\underbrace{e^{-k(T-t)\langle D \rangle^{2(1-\sigma)}_{h}}e^{-\delta_1\langle D \rangle^{\frac1\theta}_{h}}}_{\text{order zero}}e^{\delta_1\langle D \rangle^{\frac1\theta}_{h}} v_1=:v_2\in H^m_{\rho'-\delta_1;\theta}(\R),$
so 
$$
v(t,x)= \, ^{R}\{e^{-\tilde{\Lambda}}(x,D)\}\! \underbrace{\sum_{j} (-r(x,D))^{j}}_{\text{order zero}} v_2= \, ^{R}\{e^{-\tilde{\Lambda}}(x,D) \} v_3,\quad v_3\in H^m_{\rho'-\delta_1;\theta}(\R);
$$
but, by Proposition \ref{contgev}, $^{R}\{e^{-\tilde{\Lambda}}(x,D)\}$ maps $H^m_{\rho;\theta}$ spaces into $H^m_{\rho-\delta_2;\theta}$, for every $\delta_2>0$, hence we finally obtain $(\delta= \delta_1+\delta_2)$ that
$v(t,x)\in H^m_{\rho'-\delta;\theta}(\R)$ for all $\delta>0.$
We remark that the solution exhibits (an arbitrarily small) loss $\delta$ in the coefficient of the exponential weight: the solution is less regular than the Cauchy data. 
Moreover, denoting $\rho''=\rho'-\delta$, from \eqref{donkey_kong_1} we obtain that $v$ satisfies the following energy estimate 
\begin{align*}
	\|v(t)\|^{2}_{H^{m}_{\rho'';\theta}} &= \| \{e^{\Lambda}(t,\cdot,D)\}^{-1} w(t)  \|^{2}_{H^{m}_{\rho'';\theta}}  \leq C_{\rho', T} \| w(t) \|^{2}_{H^{m}} \\
	&\leq C_{\rho, T} C_{\Omega, \rho', T}  \left( \|  g_{\tilde\Lambda, k,\rho'} \|^{2}_{H^{m}}  + \int_{0}^{t} \| f_{\tilde\Lambda, k,\rho'}(\tau)\|^{2}_{H^{m}} d\tau \right) \\
	&\leq C_{\Omega, \rho', T}  \left( \| g \|^{2}_{H^{m}_{\rho;\theta}}  + \int_{0}^{t} \|f(\tau)\|^{2}_{H^{m}_{\rho;\theta}} d\tau \right) , \quad t \in [0,T].
\end{align*}

Finally, let us notice that if the data are valued in $H^{m}_{\rho;\theta}(\R)$ for every $\rho>0$, then the solution belongs to $H^{m}_{\rho'';\theta}(\R)$ for every $\rho'' \in (0,\rho)$, that is $v\in C([0,T]; H^{\infty}_{\theta}(\R))$.
\qed

\medskip
The argument of the proof of Theorem \ref{iop}, suitably simplified, provides a well-posedness result in projective Gevrey-Sobolev spaces for genuinely linear $3$-evolution equations, that is when the coefficients of the operator do not depend on $u$. Since also this result is new in the literature we state it here below as a separate result.

\begin{Cor}\label{iop2}
	Let $P$ be a linear differential operator of the form \eqref{3evolop} and assume that  $a_3 \in C([0,T]; \R)$ is such that $|a_3(t)| \geq C_{a_3} > 0$ for all $t \in [0,T]$ and for some constant $C_{a_3}$. Let moreover $\sigma \in \left(\frac12,1\right)$ and $\theta_0<\frac1{2(1-\sigma)}$ such that for $j=0,1,2$ the coefficients $a_j$ satisfy the following assumptions: for every $A>0$ there exists $C_A>0$ such that
	$$
	|\partial^{\beta}_{x} a_j(t,x)| \leq  
	C_A A^{\beta} \beta!^{\theta_0}  \langle x\rangle^{- \frac{j\sigma}2},
	$$
	for every $x \in \R, t \in [0,T]$ and $\beta \in \N_0$. Then for every $m\in\R,\ \rho>0$, $\theta\in \left[\theta_0,\frac1{2(1-\sigma)}\right)$ and  
	$f\in C([0,T]; H^m_{\rho;\theta}(\R))$, $g\in H^m_{\rho;\theta}(\R)$, there exists a unique solution
	$v\in C^1([0,T]; H^{m}_{\rho-\delta;\theta}(\R))$ for every $\delta \in (0, \rho)$ of the Cauchy problem \eqref{genCP}
	and the following energy
	estimate is satisfied:
	\beqs
	\label{enestv}
	\|v(t,\cdot)\|^2_{H^m_{\rho-\delta;\theta}}\leq C_{\rho,T}
	\left(\|g\|^2_{H^m_{\rho;\theta}}+
	\int_0^t\|f(\tau,\cdot)\|^2_{H^m_{\rho;\theta}} \,d\tau\right)\quad
	\forall t\in[0,T],
	\eeqs
	for some positive constant $C_{\rho,T}$. Moreover, if $f\in C([0,T], H^{\infty}_{\theta}(\R))$ and $g\in H^{\infty}_{\theta}(\R)$, then $v$ belongs to $C^{1}([0,T]; H^{\infty}_{\theta}(\R))$.
\end{Cor}

\section{The quasilinear problem}
\label{finalsection}

In this section we consider the quasilinear Cauchy problem \eqref{CP} and prove Theorem \ref{main}. First of all, by Theorem \ref{tameness}, it is easy to verify that the space $$X_T:=C^1([0,T];H^\infty_\theta(\R))$$ 
is a tame Fr\'echet space endowed with the family of seminorms
\beqsn
\|u\|_k=\sup_{t \in [0,T]} \big\{ |u(t,\cdot)|_k+|D_tu(t,\cdot)|_k \big\} ,
\qquad k\in\N_0,
\eeqsn
for every $\theta >1.$ Let us consider, for every $u \in X_T$, the map
\beqs\label{heart}
J(u):=&&u(t,x)-g(x)+i\int_0^ta_3(s)D_x^3u(s,x)ds+i\int_0^ta_2(s,x,u(s,x))D_x^2u(s,x)ds\\
\nonumber
&&+i\int_0^ta_1(s,x,u(s,x))D_xu(s,x)ds+i\int_0^ta_0(s,x,u(s,x))u(s,x)ds
\\\nonumber
&&-i\int_0^tf(s,x)ds.
\eeqs
\begin{Rem}
		By Lemma \ref{compositionformula}, we have $a_j(t,x,u(t,x))\xi^{j} \in C([0,T];\Gamma_{\theta}^j(\R^{2}))$, then from Pro\-po\-si\-tion \ref{prop_continuity_finite_order_gevrey_sobolev} we conclude that $a_{j}(s,x,u(s,x))D^{j}_{x}u(s,x) \in C([0,T];H^{\infty}_{\theta}(\R))$. This implies that the map $J$ maps $X_T$ into itself.
\end{Rem}
As anticipated in the introduction we shall prove the existence of a unique solution 
$u\in C^1([0,T^*];H^\infty_\theta(\R))$ for some $T^* \in (0,T]$ of the Cauchy problem \eqref{CP} by showing  the existence of a unique solution $u\in C^1([0,T^*];H^\infty_\theta(\R))$ of the integral equation
\beqs
J(u)\equiv0\ {\rm in}\ [0,T^*]\times\R. 
\eeqs
This will be achieved using Theorem \ref{NM}. The choice of rephrasing the Cauchy problem in an integral form comes from the fact that due to the presence of the time derivative, the operator $P(t,x,u, D_t,D_x)$ does not map $X_T$ into itself. 

It is not difficult to prove that $J$ is tame together with all its derivatives. To apply the Nash-Moser Theorem we only need to prove that the
equation $DJ(u)v=h$ has a unique solution $v:=S(u,h)\in X_T$ for all $u,h\in X_T$ and
that the map 
\begin{equation}\label{mapS}
	S:\ X_T\times X_T\to X_T: (u,h)\to v=S(u,h)
\end{equation} 
is smooth tame, where $DJ(u)v$ stands for the derivative of $J$ at $u$ in the direction $v$.

The next Proposition proves that the map $S$ in \eqref{mapS} is well defined.

\begin{Rem}
		We claim that $\lim_{\varepsilon \to 0} a_j(s,x, u(s,x) - \varepsilon v(s,x)) = a_j(s,x,u(s,x))$, wherehe limit is taken with respect to  the topology of $C([0,T]; H^{\infty}_{\theta}(\R))$. Indeed, first we write
		\begin{align*}
			a_j(s,x,u+\varepsilon v) - a_j(s,x,u) &= \int_{0}^{\varepsilon} \frac{d}{dr} \{ a_j(s,x,u+rv) \} dr \\
			&= v\int_{0}^{\varepsilon} {\partial_{w} a_j}(s,x,u+rv) dr 
			\\ &= v \cdot \sigma_\varepsilon(s,x) .
		\end{align*}
		Observe that $\sigma_{\varepsilon} \in C([0,T]; \tilde{\Gamma}^{0}_{\theta}(\R^2))$. Therefore, by Proposition \ref{prop_continuity_finite_order_gevrey_sobolev}, we get $\sigma_{\varepsilon}(s,x) v(s,x) \in C([0,T]; H^{\infty}_{\theta}(\R))$. Moreover, since the norms $|\sigma(s)|_{A}$ are bounded by a constant of the form $\varepsilon C_{\Omega, A}$, $\Omega$ being a bounded neighborhood of $u$, we are able to conclude 
		\begin{align*}
			|\sigma_{\varepsilon}(s,x) v|_{k} \to 0, \quad \text{as} \,\, \varepsilon \to 0,
		\end{align*}
		for every $k \in \N_0$, which finalizes the proof of our claim. In the same manner one gets  
		$$
		\lim_{\varepsilon \to 0} \frac{a_j(s,x,u(s,x) + \varepsilon v(s,x) ) - a_j(s,x,u(s,x))}{\varepsilon} = \partial_{w}a_j(s,x,u(s,x))v(s,x)
		$$
		in $C([0,T], H^\infty_\theta(\R)).$ We shall use extensively these two types of limits in the sequel.
\end{Rem}

\begin{Prop} For every $u, h\in X_T$, there exists a unique $v\in X_T$ solution of the equation $DJ(u)v=h$, and the function $v$ satisfies for every $k\in\N$ the following estimate:
	\beqs
	\label{enestsemin}
	|v(t,\cdot)|^2_{k}\leq C_{\Omega,k,T}
	\left(|h(0)|^2_{k+1}+
	\int_0^t|D_t h(\tau,\cdot)|^2_{k+1} \,d\tau\right), \quad
	\forall t\in[0,T],
	\eeqs
	where $\Omega$ stands for some bounded open neighborhood of $u$.
\end{Prop}

\begin{proof}
	By the definition \eqref{heart} of the map $J$, let us compute the
	derivative of $J$, for $u,v\in X_T$:
	\beqsn \label{cuore}
	DJ(u)v=&&\lim_{\varepsilon\to0}\frac{J(u+\varepsilon v)-J(u)}{\varepsilon}\\
	=&&\lim_{\varepsilon\to0}\Big\{v+i\int_0^ta_3(s)D_x^3v(s)ds
	+i\sum_{j=0}^{2}\int_0^t
	\frac{a_j(s,x,u+\varepsilon v)-a_j(s,x,u)}{\varepsilon}D_x^ju(s)ds\\
	&&+i\sum_{j=0}^{2}\int_0^ta_j(s,x,u+\varepsilon v)D_x^jv(s)ds\Big\}\\
	=&&v+i\int_0^t\hskip-0.25cma_3(s)D_x^3v(s)ds+i\sum_{j=0}^{2}\int_0^t\hskip-0.25cm
	(\partial_wa_j)(s,x,u)v(s)D_x^ju(s)ds \\
	&&+i\sum_{j=0}^{2}\int_0^t\hskip-0.25cma_j(s,x,u)D_x^jv(s)ds\\
	=&&v+i\int_0^t\hskip-0.25cma_3(s)D_x^3v(s)ds+i\int_0^t\hskip-0.25cm a_2(s,x,u)D_x^2v(s)ds+i\int_0^t\hskip-0.25cm a_1(s,x,u)D_x v(s)ds
	\\
	&&+i\int_0^t\hskip-0.2cm\underbrace{\Big(a_0(s,x,u)+\sum_{j=0}^{2}(\partial_wa_j)(s,x,u)D_x^ju
		\Big)}v(s)ds=:J_{0,u,0}(v) \label{cuore}
	\eeqsn
	\vskip-0.2cm\hskip+5cm$:=\tilde a_0(s,x,u)$
	\vskip+0.2cm
	where, given $u,g,f\in X_T$, the map  $J_{g,u,f}:\ X_T\to X_T$ is defined by
	\beqsn
	J_{g,u,f}v:=&&v(t,x)-g(x)+i\int_0^ta_3(s)D_x^3v(s,x)ds+i\int_0^t{a}_2(s,x,u(s,x))D_x^2v(s,x)ds
	\\
	&&+i\int_0^t{a}_1(s,x,u(s,x))D_xv(s,x)ds+i\int_0^t\tilde{a}_0(s,x,u(s,x))v(s,x)ds-i\int_0^t
	f(s,x)ds.
	\eeqsn
	Of course, $v$ solves $J_{g,u,f}(v)\equiv0$ if and only if
	it solves the linearized Cauchy problem
	\beqsn
	\begin{cases}
		\tilde{P}_u(D)v(t,x)=f(t,x)\cr
		v(0,x)=g(x),
	\end{cases}
	\eeqsn
	where $\tilde{P}_u(D)$ is obtained from $P_u(D)$ substituting $a_0$ with
	$\tilde{a}_0$.
	
	 Writing
	\beqsn
	J_{0,u,0}(v)-h=J_{0,u,0}(v)-h_0-i\int_0^tD_th(s,x)ds=J_{h_0,u,D_t h}(v)
	\eeqsn
	with $h_0:=h(0,x)$, we see that $v$ is a solution of $DJ(u)v=h$ if and only if it is a solution of $J_{h_0,u,D_t h}(v)=0$, or equivalently of the linearized Cauchy problem
	\beqs
	\label{5bis}
	\begin{cases}
		\tilde{P}_u(D)v(t,x)=D_th(t,x)\cr
		v(0,x)=h_0(x).
	\end{cases}
	\eeqs
	Summing up, the solutions to $DJ(u)v=h$  in $X_T$ coincide with the solutions to \eqref{5bis}.

	The Cauchy problem \eqref{5bis} fulfills the assumptions of Theorem \ref{iop}, indeed, on the one hand the operators $P_u(D)$ and $\tilde{P}_u(D)$ have the same coefficients but for the terms of order $0$ for which no decay assumptions are required; on the other hand, clearly $D_t h \in C([0,T]; H_\theta^\infty(\R))$ and $h_0 \in H_\theta^\infty(\R).$ We obtain by Theorem \ref{iop} a unique solution $v \in C([0,T];H_{\theta}^\infty(\R))$ of \eqref{5bis} which satisfies an energy estimate of the form \eqref{enestv} for every $\rho,\delta>0$ with $0<\delta<\rho$. Taking $\rho=k+1$ and $\delta=1$ in \eqref{enestv}, $k \in \N$, we obtain \eqref{enestsemin}.
\end{proof}

\begin{Lemma}
	The map $S$ defined by \eqref{mapS} is smooth tame.	
\end{Lemma}
\begin{proof}We have to prove that $S$ and its derivatives $D^m S$ are tame maps for any positive integer $m$. Let us first prove that $S$ is tame. First of all, notice that if we take $u$ in a bounded set $\Omega \subset X_T$, from \eqref{enestsemin} we get
	\begin{equation} \label{quadratino}
	\sup_{t \in [0,T]}|v(t, \cdot)|_k \leq C_{\Omega, k, T} \|h\|_{k+1}
	\end{equation}
	for every $k \in \N$ and for some $C_{\Omega,k,T}>0$. Moreover, from the equation it follows that
	\begin{align*}
		|D_t v(t,\cdot)|_k &= \left| -a_3(t)D_x^3 v(t,\cdot) - \sum_{j=1}^2 a_j (t, \cdot, u)D_x^j v(t,\cdot) + \tilde{a}_{0}(t,\cdot,u)v(t,\cdot) + D_t h(t,\cdot) \right|_k  \\
		&\leq C(|v(t,\cdot)|_{k+1} + \|h\|_k)
	\end{align*}
	for some $C>0$ depending on the set $\Omega$ and the coefficients. Hence 
	\beqs
	\label{9bis}
	\|S(u,h)\|_k = \sup_{t\in[0,T]}\left(|v(t,\cdot)|_k+|D_t v(t,\cdot)|_k
	\right)
	\leq C_{\Omega, k,T} \|h\|_{k+1}\leq C_{\Omega, k, T} \|(u,h)\|_{k+1}
	\eeqs
	for some (possibly larger than before) constant $C_{\Omega, k, T}>0$, and so $S$ is tame.
	\\
	Let us now consider the first derivative of $S$, defined for $(u,h), (u_1,h_1)\in X_T \times X_T $ as
	$$DS(u,h)(u_1,h_1) = \lim_{\varepsilon \to 0} \frac{S(u+\varepsilon u_1, h+\varepsilon  h_1)-S(u,h)}{\varepsilon} = \lim_{\varepsilon \to 0} \frac{v_\varepsilon -v}{\varepsilon}=\lim_{\varepsilon \to 0} w_\varepsilon,$$ where $w_\varepsilon :=\varepsilon^{-1}(v_\varepsilon-v)$ and $v_\varepsilon=S(u+\varepsilon u_1, h+\varepsilon  h_1)$ is the solution of the Cauchy problem \begin{equation}\label{pe} \begin{cases} \tilde{P}_{u+\varepsilon u_1} (D) v=D_t (h+\varepsilon h_1) \\ v(0,x) = h(0,x)+\varepsilon h_1(0,x). \end{cases}\end{equation}
	Since $v_\epsilon, v$ solve the Cauchy problems \eqref{5bis} and \eqref{pe} respectively, it is easy to check that the function $w_\varepsilon$ satisfies 
	\begin{equation}\label{sea}\begin{cases} 
			\tilde{P}_{u+\varepsilon u_1} w_\varepsilon =  f_\varepsilon\\ w_\varepsilon (0,x)= h_1(0,x) \end{cases}
	\end{equation}
	with (omitting $(t,x)$ in the notation for brevity's sake) 
	$$
	f_\varepsilon :=D_t h_1-  \frac{{a}_2(u+\varepsilon u_1)-a_2(u)}\varepsilon D_x^2v-\frac{a_1(u+\varepsilon u_1)-a_1(u)}\varepsilon D_x v-\frac{\tilde{a}_0(u+\varepsilon u_1)-\tilde{a}_0(u)}\varepsilon v.
	$$
	
	If we prove that the sequence $\{w_\varepsilon\}_\varepsilon$ is a Cauchy sequence in $X_T$, then we obtain that $w_\varepsilon$ converges to some $w$ in $X_T$; this function $w$, which is on one hand the first derivative of $S$, is on the other hand the solution to
	\beqsn
	\begin{cases}
		\tilde{P}_u(D)w=f_1\cr
		w(0,x)=h_1(0,x)
	\end{cases}
	\eeqsn
	with
	\beqsn
	f_1:=\lim_{\varepsilon\to0}f_\varepsilon=D_th_1-
	\partial_w a_2(u)u_1D_x^2v-
	\partial_w a_1(u)u_1D_xv-
	\partial_w\tilde{a}_0(u)u_1v,
	\eeqsn
	so, taking $u$ in a bounded set $\Omega$, by Theorem~\ref{iop} it satisfies the energy estimate
	\beqsn
	|w(t,\cdot)|_k^2\leq&& C_{\Omega, k, T} \Big(|h_1(0,\cdot)|_{k+1}^2+
	\int_0^t|f_1(\tau,\cdot)|_{k+1}^2d\tau\Big).
	\eeqsn
	Now if we take $u_1$ in a bounded set $\Omega_1$, by \eqref{quadratino} we get 
	\beqsn
	\nonumber
	|w(t,\cdot)|_k&&\leq C_{\Omega, k,T} (|h_1(0,\cdot)|_{k+1} + \sup_{t \in [0,T]} |f_1(t,\cdot)|_{k+1})
	\\\nonumber
	&&\leq C_{\Omega, \Omega_1, k, T} \sup_{t\in[0,T]}\left(|h_1(t,\cdot)|_{k+1}+
	|D_th_1(t,\cdot)|_{k+1}
	+|v(t,\cdot)|_{k+2}\right)\\
	&&\leq C_{\Omega, \Omega_1, k, T} \left(\|h_1\|_{k+1}+\|h\|_{k+3}\right)
	\eeqsn
	for some positive constant $C_{\Omega, \Omega_1, k, T}$ depending on $\Omega_1, \Omega_2, k, T$ and the coefficients. Also 
	$$
	D_tw=-a_3(t)D_x^3w-a_2(t,x,u)D_x^2w-a_1(t,x,u)D_xw-\tilde{a}_0(t,x,u)w+f_1
	$$
	satisfies a similar estimate, so the first derivative $DS$ (coinciding with $w$) is tame. 
	\\
	Thus, we only need to prove that $\{w_\varepsilon\}_{\varepsilon \in [0,1]}$ is a Cauchy sequence in $X_T$ to conclude that $DS$ is a tame map. 
	\\
	To this aim, arguing as before, let us consider $w_{\varepsilon_1}$ and $w_{\varepsilon_2}$ solutions of the Cauchy problems
	$$\tilde{P}_{u+\varepsilon_iu_1}(D)w_{\varepsilon_i}=f_{\varepsilon_i},\qquad
	w_{\varepsilon_i}(0,x)=h_1(0,x),\quad i=1,2;$$
	then $w_{\varepsilon_1}-w_{\varepsilon_2}$ solves
	\beqsn\begin{cases}
		\tilde{P}_{u+\varepsilon_1u_1}(D)(w_{\varepsilon_1}-w_{\varepsilon_2})
		=f_{\varepsilon_1}-f_{\varepsilon_2}+f_{\epsilon_1,\epsilon_2}\cr
		(w_{\varepsilon_1}-w_{\varepsilon_2})(0,x)=0
	\end{cases}
	\eeqsn
	with (omitting $(t,x)$ in the notation) 
	\beqsn f_{\epsilon_1,\epsilon_2}:&=&\big(a_2(u+\varepsilon_2u_1)-
	a_2(u+\varepsilon_1u_1)\big)D_x^2
	w_{\varepsilon_2}
	\\
	&+&\big(
	a_1(u+\varepsilon_2u_1)-a_1(u+\varepsilon_1u_1)\big)D_x
	w_{\varepsilon_2}+\big(
	\tilde{a}_0(u+\varepsilon_2u_1)-
	\tilde{a}_0(u+\varepsilon_1u_1)\big)w_{\varepsilon_2}\eeqsn
	and the energy estimate \eqref{enestv} gives
	\beqs\label{cau}
	|(w_{\varepsilon_1}-w_{\varepsilon_2})(t,\cdot)|_k\leq C_{\Omega, \Omega_1, k, T}
	&&\sup_{t\in[0,T]}\bigg(|f_{\varepsilon_1}(t,\cdot)
	-f_{\varepsilon_2}(t,\cdot)|_{k+1}+ |f_{\epsilon_1,\epsilon_2}(t,\cdot)|_{k+1}\bigg).
	\eeqs
	By Lagrange theorem, there exist $\tilde{u}_j, j=0,1,2,$
 between $u+\varepsilon_1u_1$ and $u+\varepsilon_2u_1$ such that,
	for all $t\in[0,T]$,
	\beqsn
	|f_{\epsilon_1,\epsilon_2}(t,\cdot)|_{k+1}&\leq&|\varepsilon_1-\varepsilon_2|\sup_{t\in [0,T]}\left(
	|\partial_wa_2(t,\cdot,\tilde{u}_{2})u_1(t,\cdot)D_x^2 w_{\varepsilon_2}(t,\cdot)|_{k+1}\right.
	\\
	&&\left.+|\partial_wa_1(t,\cdot,\tilde{u}_1)u_1(t,\cdot)D_x
	w_{\varepsilon_2}(t,\cdot)|_{k+1}+|\partial_w a_0(t,\cdot,\tilde{u}_0) u_1(t,\cdot)w_{\varepsilon_2}(t,\cdot)|_{k+1}\right)
	\\
	&\leq&|\varepsilon_1-\varepsilon_2|C_{\Omega, \Omega_1, k, T}
	| w_{\varepsilon_2}|_{k+2}
	\eeqsn
	with $C_{\Omega, \Omega_1, k, T}$ independent of $\varepsilon_1, \varepsilon_2 \in [0,1]$, where we used the algebra property of $H^\infty_\theta$ spaces: namely, we know that $H^m_{\rho,\theta}(\R^n)$ is an algebra if $m>n/2$, see for instance \cite{CicZan}. Hence, for every $f,g \in H^{\infty}_{\theta}$ we have $fg \in H^{\infty}_{\theta}$ and, taking $m > \frac{n}{2}$ we may write
		$$
		\|fg\|_{H^{0}_{\rho; \theta}} \leq \| fg \|_{H^{m}_{\rho, \theta}} \leq C \|f \|_{H^{m}_{\rho;\theta}} \| g \|_{H^{m}_{\rho;\theta}} \leq C_2 \|f\|_{H^{0}_{\rho+\varepsilon;\theta}} \|g\|_{H^{0}_{\rho +\varepsilon;\theta}}
		.
		$$
	From the energy inequality for the linearized problem we see that $|w_{\varepsilon}|_{k}$ is bounded with respect to $\varepsilon \in [0,1]$ for every $k \in \N_0$. Hence $f_{\epsilon_1,\epsilon_2}\to 0$ as $\varepsilon_1, \varepsilon_2 \to 0$ in the $H^{\infty}_{\theta}(\R)$ topology. In the same manner one gets $f_{\varepsilon_1}-f_{\varepsilon_2}\to 0$ as $\varepsilon_1, \varepsilon_2 \to 0$.

	This gives that $\{w_{\varepsilon}\}_\varepsilon$ is a Cauchy sequence in $X_T$ and therefore  we can concludes that $DS$ is a tame map.
	To conclude the proof it is sufficient to repeat the previous computations in an inductive procedure similar to the one in the proof of \cite[Theorem 1.3, Step 4]{ABbis}.
\end{proof}

We are now ready for the final step of this paper, that is the proof of the main Theorem \ref{main}.

\begin{proof}[Proof of Theorem \ref{main}]
	As described at the beginning of this Section the existence of a unique local solution
	$u\in C^1([0,T^*];H^\infty_\theta(\R))$ of the Cauchy problem \eqref{CP} is
	equivalent to the existence of a unique solution $u\in C^1([0,T^*];H^\infty_\theta(\R))$ of the equation
	\beqs
	\label{defu}
	u(t,x)&=&g(x)-i\int_0^ta_3(s)D_x^3u(s,x)ds-i\int_0^ta_2(s,x,u(s,x))D_x^2u(s,x)ds\\
	\nonumber
	&&-i\int_0^ta_1(s,x,u(s,x))D_xu(s,x)ds-i\int_0^ta_0(s,x,u(s,x))u(s,x)ds
	\\\nonumber
	&&+i\int_0^tf(s,x)ds.
	\eeqs
	Equation \eqref{defu} provides the first order Taylor expansion of  $u$:
	\beqs\nonumber
	u(t,x)&=&g(x)-it\left(a_3(0)D_x^3g(x)+a_2(0,x,g(x))D_x^2g(x)\right.
	\\\nonumber
	&+&\left.a_1(0,x,g(x))D_xg(x)+a_0(0,x,g(x))g(x)-f(0,x)\right)+o(t)
	\\\label{defw}
	&=:&w(t,x)+o(t),\qquad \mbox{as}\ t\to0.
	\eeqs
	If $t$ is sufficiently small, the function $w\in X_T$ is in a neighborhood of the solution $u$
	we are looking for. The idea of the proof is then the following: we first
	approximate $Jw\in X_T$ by a function $\phi_\varepsilon$ such that $\phi_\varepsilon(t)\equiv 0$ for
	$0\leq t\leq T_\varepsilon\leq T$; then, we apply
	Theorem \ref{NM}, in particular the fact that
	$J$ is a bijection of a
	neighborhood $U$ of $w$ onto a neighborhood $V$ of $Jw$. More precisely, 
	we show that $\phi_\varepsilon\in V$, and then by the local invertibility of $J$ there will be $u\in U$ such that
	$Ju=\phi_\varepsilon\equiv0$ in $[0,T_\varepsilon]$ and hence $u$ is the local in time solution of the Cauchy problem \eqref{CP}.
%
	Let us construct $\phi_\varepsilon$: given $\rho\in C^\infty(\R)$ with $0\leq\rho\leq1$ and
	\beqsn
	\rho(s)=\begin{cases}
		0, &s\leq1\cr
		1, &s\geq2,
	\end{cases}
	\eeqsn
	we define
	\beqsn
	\phi_\varepsilon(t,x):=
	\int_0^t\rho\left(\frac s\varepsilon\right)(\partial_tJw)(s,x)ds.
	\eeqsn
	We immediately see that $\phi_\varepsilon\equiv0$ for $0\leq t\leq\varepsilon$. We are going to prove
	that, for every fixed neighborhood $V$ of $Jw$ in the topology of
	$X_T=C^1([0,T];H^\infty(\R))$, we have $\phi_\varepsilon\in V$ if
	$\varepsilon$ is sufficiently small. To this aim let us notice that 
	\begin{align*}
		Jw(t) - \phi_{\varepsilon}(t) &= \int_0^{t} (\partial_{t}Jw)(s) ds + \underbrace{Jw(0)}_{= 0} - \int_{0}^{t} \rho \left(\frac s\varepsilon\right)(\partial_tJw)(s)ds \\
		&= \int_{0}^{t} \left[1- \rho \left(\frac s\varepsilon \right) \right] (\partial_t Jw)(s)ds.
	\end{align*}
	Hence 
	\begin{equation} \label{norm_estimate}
		\| Jw - \phi_{\varepsilon} \|_{k} \leq \int_{0}^{2\varepsilon} | (\partial_{t}Jw)(s)|_{k}\, ds + 
		\sup_{t \in [0,2\varepsilon]} |(\partial_{t}Jw)(t)|_{k}.
	\end{equation}
	Now we compute explicitly $\partial_t (Jw(t,x))$ and estimate its $k-$seminorms for small values of $t$.
	From \eqref{heart} we get
	\beqsn
	\partial_t (Jw(t,x))
	=&&\partial_tw+ia_3(t)D_x^3w+\sum_{j=0}^2ia_j(t,x,w)D_x^jw-if(t,x).
	\eeqsn
	Using the definition \eqref{defw} of $w$ we get
	\beqsn
	\partial_t (Jw(t,x))&=&-ia_3(0)D_x^3g-\sum_{j=0}^2ia_j(0,x,g)D_x^jg+if(0,x)
	\\
	&&+ia_3(t)D_x^3g+ta_3(t)D_x^3\Big(a_3(0)D_x^3g+\sum_{j=0}^2a_j(0,x,g)D_x^jg
	-f(0,x)\Big)
	\\
	&&+ia_2(t,x,w)D_x^2g+ta_2(t,x,w)D_x^2\Big(a_3(0)D_x^3g+\sum_{j=0}^2a_j(0,x,g)D_x^jg
	-f(0,x)\Big)
	\\
	&&+ia_1(t,x,w)D_xg+ta_1(t,x,w)D_x\Big(a_3(0)D_x^3g+\sum_{j=0}^2a_j(0,x,g)D_x^jg
	-f(0,x)\Big)
	\\
	&&+ia_0(t,x,w)g+ta_0(t,x,w)\Big(a_3(0)D_x^3g+\sum_{j=0}^2a_j(0,x,g)D_x^jg
	-f(0,x)\Big)
	\\
	&&-if(t,x)\\
	=&&i[a_3(t)-a_3(0)]D_x^3g
	+i\sum_{j=0}^{2}\big[a_j(t,x,w)-a_j(0,x,g)\big]D_x^jg\\
	&&+a_3(t)tD_x^3\bigg[a_3(0)D_x^3g+\sum_{j=0}^{2}
	a_j(0,x,g)D_x^jg-f(0,x)\bigg]\\
	&&+\sum_{j=0}^{2}a_j(t,x,w)tD_x^j\bigg[a_3(0)D_x^3g+\sum_{s=0}^{2}
	a_s(0,x,g)D_x^sg-f(0,x)\bigg]\\
	&&+i\big(f(0,x)-f(t,x)\big).
	\eeqsn
	Now observe that for every $k \in \N_0$ we have:
	\begin{itemize}
		
		\item for every $\varepsilon_3 > 0$ there exists $\delta_3 > 0$ depending on $g, a_3$ and $k$ such that \\ $|(a_3(t) - a_3(0))D^{3}_{x}g|_{k} \leq \varepsilon_3$ for every $t \in [0,\delta_3]$,  since $a_3$ is continuous;
		
		\item for every $\varepsilon_2 > 0$ there exists $\delta_2 > 0$ depending on $g, a_0, a_1, a_2$ and $k$ such that 
		$$
		\ds\sum_{j=0}^{2}|\left(a_j(t,x,w)-a_j(0,x,g)\right)D_x^jg|_k \leq \varepsilon_2
		$$ 
		for every $t \in [0,\delta_2]$, since $(t,x) \mapsto a_j(t,x,w(t,x))$ belongs to $C([0,T]; \gamma^{\theta}(\R))$ and henceforth $(t,x) \mapsto a_j(t,x,w(t,x))\xi^{j}$ belongs to $C([0,T];\Gamma_{\theta}^j(\R^{2}))$;
		
		\item $t\ds\sup_{t\in[0,T]}|a_3(t)|\cdot|a_3(0)D_x^6g+\sum_{j=0}^{2}D_x^3(a_j(0,x,g)D_x^jg)-D_x^3f(0,x)|_{k}\leq C(a_3,a_2,a_1,a_0,g,f,k)t$;
		
		\item $t\ds\sum_{j=0}^{2}\left|a_j(t,x,w)D_x^j\left(a_3(0)D_x^3g+
		\sum_{s=0}^{2}a_s(0,x,g)D_x^sg-f(0,x)\right)\right|_k\leq C(a_3,a_2,a_1,a_0,g,f,k)t;$
		
		\item for every $\epsilon_1 > 0$ there exist $\delta_1 > 0$ depending on $f$ and $k$ such that $|f(0,x)-f(t,x)|_k \leq \varepsilon_1$ for every $t \in [0,\delta_1]$, since $f \in C([0,T]; H^{\infty}_{\theta}(\R))$.
		
	\end{itemize}
	In this way we are able to conclude that for every $\tilde{\varepsilon} > 0$ there exists $\tilde{\delta}_{k}$ depending on $k, f, g$ and on the coefficients such that 
	$$
	|(\partial_t Jw)(t)|_{k} \leq \tilde{\varepsilon}, \quad \forall\, t \in [0, \tilde{\delta}_{k}].
	$$ 
	If $2\varepsilon < \tilde{\delta}_k$, from \eqref{norm_estimate} we obtain
	$$
	\|Jw - \phi_\varepsilon \|_k \leq 2\tilde{\varepsilon}.
	$$
	
	Now let $V$ be an open neighborhood of $Jw$. Recalling that the topology of $X_T$ is given by the metric
	$$
	d(u,v) = \sum_{k \geq 0} \frac{1}{2^{k+1}} \frac{\|u-v\|_k}{1+\|u-v\|_k}, \quad \forall u,v \in X_T,
	$$
	we see that there exists $r > 0$ such that the ball $\{ d(u,Jw) < r : u \in X_T\} \subset V$. Take $K \in \N_0$ such that $\sum_{k > K} 2^{-(k+1)} \leq \frac{r}{2}$ and choose $\tilde{\varepsilon} > 0$ such that $\tilde{\varepsilon} < \frac{r}{4K}$. Then, if $2\varepsilon < \tilde{\delta} = \ds\min_{0\leq k \leq K} \tilde{\delta}_{k}$ we infer that 
	\begin{align*}
		d(Jw, \phi_{\varepsilon}) \leq \sum_{k \leq K} \|Jw - \phi_{\varepsilon}\|_{k} + \sum_{k > K} \frac{1}{2^{k+1}} \leq \frac{r}{2} + \frac{r}{2} = r.
	\end{align*}
	So, if $\varepsilon > 0$ is sufficiently small then $\phi_\varepsilon \in V$. 
	
	
	Now suppose in addition that $V=J(U)$ where $U$ is an open neighborhood of $w$ and that $J:U \rightarrow V$ is bijective. Then there exists $u \in U$ such that $Ju = \phi_{\varepsilon}$.
	In particular, this proves that $u\in C^1([0,\varepsilon];H^\infty_\theta(\R))$ is a
	local solution of the Cauchy problem \eqref{CP}.
	Uniqueness follows by standard arguments. Indeed, if $u,v$ are two solutions of the Cauchy problem \eqref{CP}, then $w:=u-v$ solves the linear Cauchy problem
	$$
	\tilde{\tilde P}w=0,\qquad 
	w(0,x)=0,
	$$
	for an operator $\tilde{\tilde P}$ which is exactly as $P_u(D)$ except for the term $a_0$ which is substituted by another term satisfying the same assumptions. From the uniqueness of the solution to the linearized Cauchy problem we get $w=0$, that is $u=v$.
\end{proof}

\begin{Rem}\label{xnelleadingterm}
	In Theorem \ref{main} we assume that the coefficient of the third order term is indepen\-dent of $x$. In the $H^\infty$ setting, it is possible to consider also the more general case $a_3(t,x)D_x^3$, assuming for $a_3$ suitable decay estimates for $|x| \to \infty$, see \cite[Section 4]{ABbis}.  This is not possible in the Gevrey setting, due to the conjugation with $e^{\rho' \langle D \rangle^{1/\theta}}$; indeed, if $a_3$ depends on $x$, even allowing its derivatives with respect to $x$ to decay like $\langle x \rangle^{-m}$ for $m>>0$, we obtain
	$$e^{\rho' \langle D \rangle^{1/\theta}} (ia_3(t,x)D_x^3) e^{-\rho' \langle D \rangle^{1/\theta}} = ia_3(t,x)D_x^3 + \textrm{op }\left(\rho' \partial_\xi \langle \xi \rangle^{\frac{1}{\theta}}\cdot \partial_x a_3 (t,x)\xi^3\right) + \textrm{l.o.t}$$
	with $\rho' \partial_\xi \langle \xi \rangle^{\frac{1}{\theta}}\cdot \partial_x a_3 (t,x)\xi^3 \sim \langle \xi \rangle^{2+\frac1{\theta}} \langle x \rangle^{-m}.$ This term has order $2+\frac{1}{\theta} >2$ and cannot be controlled by other lower order terms whose order does not exceed $2$.
\end{Rem}

\end{document}